\documentclass[11pt]{article}
\usepackage{amsmath}
\usepackage{amsthm}
\usepackage{amssymb}
\usepackage{amscd}
\usepackage{relsize}
\usepackage{enumerate}

\newtheorem{teo}{Theorem}
\newtheorem{prop}{Proposition}
\newtheorem{cor}[teo]{Corollary}
\newtheorem{lema}[teo]{Lemma}

\newcommand{\CC}{\mathbb{C}}
\newcommand{\RR}{\mathbb{R}}
\newcommand{\ZZ}{\mathbb{Z}}
\newcommand{\NN}{\mathbb{N}}
\newcommand{\Ac}{\mathcal{A}}
\newcommand{\Bc}{\mathcal{B}}
\newcommand{\Hc}{\mathcal{H}}
\newcommand{\Tc}{\mathcal{T}}
\newcommand{\Lc}{\mathcal{L}}
\newcommand{\Dc}{\mathcal{D}}

\DeclareMathOperator*{\esup}{ess\,sup}
\DeclareMathOperator*{\einf}{ess\,inf}

\title{{\bf Modeling sampling in tensor products of unitary invariant subspaces}}
\author{
{\bf Antonio~G. Garc\'{\i}a}\thanks{E-mail:\texttt{agarcia@math.uc3m.es}}\,, \,\,
{\bf Alberto Ibort}\thanks{E-mail:\texttt{albertoi@math.uc3m.es}} \,\, 
{\bf and \, Mar\'{\i}a~J. Mu\~ noz-Bouzo}\thanks{E-mail:\texttt{mjmunoz@mat.uned.es}}. \,\,}
\date{}
\begin{document}
\maketitle
\begin{itemize}
\item[*\dag] Departamento de Matem\'aticas, Universidad Carlos III de Madrid,
 Avda. de la Universidad 30, 28911 Legan\'es-Madrid, Spain.
\item[\ddag] Departamento de Matem\'aticas Fundamentales, U.N.E.D., Senda del Rey 9, 28040 Madrid, Spain.
\end{itemize}
\begin{abstract}
The use of unitary invariant subspaces of a Hilbert space $\Hc$ is nowadays a recognized fact in the treatment of sampling problems. Indeed, shift-invariant subspaces of $L^2(\RR)$ and also periodic extensions of finite signals are remarkable examples where this occurs. As a consequence, the availability of an abstract unitary sampling theory becomes a useful tool to handle these problems. In this paper we derive a sampling theory for tensor products of unitary invariant subspaces. This allows to merge the cases of finitely/infinitely generated unitary invariant subspaces formerly studied in the mathematical literature;  it also allows to introduce  the several variables case. As the involved samples are identified as frame coefficients in suitable tensor product spaces, the relevant mathematical technique is that of frame theory, involving both, finite/infinite dimensional cases.
\end{abstract}

\noindent{\bf AMS}: 42C15; 94A20; 15A09.

\section{Introduction}
\label{section0}
Sampling and reconstruction of functions in unitary invariant subspaces of a  separable Hilbert space brings a unified approach to sampling problems (see Refs. \cite{hector:13,hector:14,garcia:15,michaeli:11,pohl:12}). Indeed, it englobes the most usual sampling settings such as sampling in shift-invariant subspaces of $L^2(\RR)$ (see, for instance, Refs. \cite{aldroubi:02,aldroubi:01,ole:03,garcia:12,garcia:06,garcia:08,kang:11,qsun:10,sun:02,  sun:03,sun:99} and references therein), or sampling periodic extensions of finite signals (see Refs. \cite{hector:15,garcia:15}). 

In a recent paper \cite{garcia3:15} it was shown how to extend sampling reconstruction theorems to invariant subspaces of a separable Hilbert spaces under a unitary representation of finite groups which are semidirect products with an Abelian factor.   This setting being appropriate for applications of the theory beyond the domain of classical telecommunications to quantum physics.   

In this paper we go one step ahead by enlarging
the class of target spaces for sampling: we deal with tensor products of different unitary invariant subspaces. This situation corresponds for instance to consider multichannel systems in classical telecommunications or composite systems in the case of quantum applications.  Thus, in this setting, we are able to gather problems of diverse nature by means of a simple formalism involving tensor products and tensor operators in Hilbert spaces. Namely, we first consider an infinite $U$-unitary subspace 
\[
\Ac_a=\big\{\sum_{n\in \ZZ} a_n\, U^n a \,:\, \{a_n\}\in \ell^2(\ZZ)\big\}
\]
in a Hilbert space $\Hc_1$, and a finite $V$-unitary subspace 
\[
\Ac_b=\big\{\sum_{p=0}^{N-1} b_p\, V^p b \, : \, b_p\in \CC\big\}
\]
in a Hilbert space $\Hc_2$, to finally obtain sampling formulas in its tensor product $\Ac_{a,b}:=\Ac_a \otimes \Ac_b$.

Apart from tensor products in Hilbert spaces, the paper involves the theory of frames.  Concretely, in this situation, the generalized samples will be expressed as frame coefficients in an auxiliary Hilbert space $L^2(0,1)\otimes \ell^2_N(\ZZ)$, where $\ell^2_N(\ZZ)$ denotes the space of all $N$-periodic  complex sequences.  Continuing the line of inquiry of \cite{hector:14,garcia:15}, the problem reduces to find appropriate families of dual frames. By `appropriate" we mean that these dual frames have a nice structure taking care of the unitary invariance of the involved sampling subspaces as it will be discussed in the sequel.

Later on, the infinite-infinite and finite-finite generator cases will be considered too.   That is, the situation where both invariant subspaces are generated by a sequence of vectors $U^na$, $V^pb$, $n,p \in \ZZ$, and the simpler case where both subspaces are finite dimensional.   Relevant examples of each situation will be discussed in detail.   


The paper is organized as follows: For the sake of completeness we include in Section \ref{section1} the basics of frames and tensor products needed in the sequel. In Section \ref{section2} we focus on the  above case that we call infinite-finite generators case; that is, we establish sampling formulas in the tensor product of two unitary invariant subspaces, one of them $\Ac_a$ with an infinite generator $a$ and the other one $\Ac_b$ finitely generated. First, we obtain appropriate expressions for the samples of any $x\in \Ac_{a,b}$ obtained from $S=ss'$ systems $\Lc_{jj'}$ acting on $\Ac_{a,b}$,
\[
\Lc_{jj'}x(rn,\bar{r}m):=\big\langle x, U^{rn}h_{j1}\otimes V^{\bar{r}m} h_{j'2}\big\rangle_{\Hc_1\otimes \Hc_2}\,,
\]
where $n\in \ZZ$, $m=0, 1,\dots, \ell-1$, $j=1,2,\dots, s$ and $j'=1,2,\dots, s'$.  Here,  $h_{j1}$, $j=1,2, \dots,s$, denote $s$ fixed elements in $\Hc_1$ and $h_{j'2}$, $j'=1,2,\dots, s'$, denote $s'$ fixed elements in $\Hc_2$,  ; $r$ and $\bar{r}$ the sampling periods, where $r\in \NN$ and $\bar{r}$ is a divisor  of $N$ and $\ell:=N/\bar{r}$ (see Section \ref{section2} for the details).
Then we state the suitable isomorphism $\Tc^{UV}_{a,b}$ between $L^2(0,1)\otimes \ell^2_N(\ZZ)$ and $\Ac_{a,b}$ which allows to transform the derived frame expansions in $L^2(0,1)\otimes \ell^2_N(\ZZ)$ into stable sampling formulas in  
$\Ac_{a,b}$ having the form 
\[
x=\sum_{j=1}^s \sum_{j'=1}^{s'}\sum_{n\in \ZZ} \sum_{m=0}^{\ell-1} \Lc_{jj'}x(rn,\bar{r}m)\, \big(U^{rn} \otimes V^{\bar{r}m} \big)\, ( c_j\otimes d_{j'})\,,
\]
where $c_j \otimes d_{j'} \in \Ac_{a,b}$, $j=1, 2, \dots, s$ and $j'=1,2,\dots, s'$. We conclude the section giving a representative example arising from classical tomography (see, for instance, Refs. \cite{natterer:01,stark:93}). 

Sections \ref{section3} and \ref{section4} deal with the called infinite-infinite and finite-finite cases. They mimic the structure of  Section \ref{section2} with auxiliary spaces $L^2(0,1)\otimes L^2(0,1)$ and $\ell^2_N(\ZZ)\otimes \ell^2_M(\ZZ)$, respectively.

Finally, it is worth to mention that for the sake of simplicity we only deal with the tensor product of two single generated unitary invariant subspaces; the same results apply for  the tensor product of any finite number of multiple generated unitary invariant subspaces.

\section{A brief on frames and tensor products}
\label{section1}
The frame concept was introduced by Duffin and Shaeffer  in \cite{duffin:52} while studying some problems in nonharmonic Fourier series; some years later it was revived by Daubechies, Grossman and Meyer in  \cite{daubechies:85}. Nowadays, frames have become a tool in pure and applied  mathematics, computer science, physics and engineering used to derive redundant, yet stable decompositions of a signal for analysis or transmission, while also promoting sparse expansions. Recall that a sequence $\{x_n\}$ is a 
frame for a separable Hilbert space $\Hc$ if there exist two constants $A,B>0$ (frame bounds) 
such that
\[
A\|x\|^2 \leq \sum_n |\langle x, x_n \rangle|^2 \leq B \|x\|^2 \,\, \text{ 
for all } x\in \Hc \,.
\]
A sequence satisfying only the right-hand inequality is said to be a Bessel sequence for  $\Hc$.
Given a frame $\{x_n\}$ for $\Hc$ the representation 
property of any vector $x\in \Hc$ as a series $x=\sum_n c_n\, x_n$ is retained, 
but, unlike the case of Riesz bases, the uniqueness 
of this representation (for overcomplete frames) is sacrificed. 
Suitable frame coefficients 
$c_n$ which depend continuously and linearly on $x$ 
are obtained by using dual frames $\{y_n\}$ of $\{x_n\}$, i.e., 
$\{y_n\}$ is another frame for $\Hc$ such that for each $x\in \Hc$
\begin{equation}
\label{dual}
x=\sum_n \langle x, y_n \rangle \,x_n=\sum_n \langle x, x_n \rangle \,y_n \quad \text{in $\Hc$}\,.
\end{equation} 
Recall that a Riesz basis in a separable Hilbert space
$\Hc$ is the image of an orthonormal basis by means of a boundedly
invertible operator; it is a particular case of frame: the so called  exact frame. Any Riesz basis $\{x_n\}$ has a
unique biorthonormal (dual) Riesz basis $\{y_n\}$,
i.e., $\langle x_n , y_m \rangle_{\Hc}=\delta_{n,m}$, such that expansion \eqref{dual} holds for every $x \in \Hc$. A Riesz sequence is a Riesz basis for its closed span. For more details on frames and Riesz bases theory see, for instance, the monograph \cite{ole:03} and 
references therein; see also Ref.~\cite{casazza:14} for finite frames.

Traditionally, frames were used in signal and image processing, nonharmonic analysis, data compression, and sampling theory, but nowadays
frame theory plays also a fundamental role in a wide variety of problems in both pure  and applied  mathematics, computer science, physics and engineering. The redundancy of frames, which gives flexibility and robustness, is the key to their significance for applications; see, for instance, the nice introduction in Chapter 1 of Ref.~\cite{casazza:14} and the references therein. 

\bigskip

Next we briefly recall some basic facts about tensor products of Hilbert spaces which will be useful in the current work.
Let $\Hc_1$, $\Hc_2$ be two  Hilbert spaces.  Among the different ways of constructing tensor product spaces we adopt the model proposed in  \cite{folland:95}. There, the tensor product $\Hc_1\otimes\Hc_2$ is defined as the space of all antilinear maps $A\colon\Hc_2\longrightarrow\Hc_1$ such that $\sum_i\|A e_i\|^2 <\infty$ for some orthonormal basis of $\{e_i\}_i$ of  $\Hc_2$.  

As for every $A, B\in \Hc_1\otimes\Hc_2$ the series $\sum_i\|A e_i\|^2 $ and $\sum_i  \langle A e_i, B e_i \rangle$ are independent of the orthonormal  basis $\{e_i\}_i$ for $\Hc_2$, then $\Hc_1\otimes\Hc_2$ can be turned into an inner product space  by defining the norm $\|A\|^2= \sum_i\|A e_i\|^2 $ (and the associated inner product $\langle A, B \rangle=\sum_i  \langle A e_i, B e_i \rangle$). Indeed $\Hc_1\otimes\Hc_2$ endowed  with this inner product  becomes a Hilbert space. 

If $u \in \Hc_1$ and $v\in  \Hc_2$  the tensor product $u\otimes v\in \Hc_1\otimes\Hc_2$ is defined to be the rank one map such that $(u\otimes v) (w)= \langle v, w \rangle u$ for every $w \in \Hc_2$. 

Let $\Bc(\Hc_1)$, $\Bc(\Hc_2)$ and $\Bc(\Hc_1\otimes \Hc_2)$ denote the spaces of all    bounded linear operators on $\Hc_1$, $\Hc_2$ and $\Hc_1\otimes\Hc_2$, respectively. If $S\in\Bc(\Hc_1)$ and $T \in \Bc(\Hc_2)$  the tensor product $S\otimes  T \in \Bc(\Hc_1\otimes \Hc_2)$  is defined to be the bounded linear operator on $\Hc_1\otimes \Hc_2$  such that $(S\otimes  T) (A)= S A T^*$ for every $A \in \Hc_1\otimes \Hc_2$.

For further information  on tensor products of Hilbert spaces see \cite{bourou:08,folland:95, khosravi:03, kubrusly: 06}. In these sources are to be found the following results for tensor products needed in the sequel:

\medskip

{\it
\noindent {\scriptsize $\bullet$} $\|u\otimes v\|= \|u\|\ \|v\|$ and 
$\langle u \otimes v , u' \otimes v' \rangle= \langle u, u' \rangle\, \langle v, v' \rangle$ for any 
$u, u' \in \Hc_1 $ and $v, v' \in \Hc_2 $.

\smallskip

\noindent {\scriptsize $\bullet$}  $\| S\otimes  T\|= \|S \| \ \|T\|$ and $(S\otimes  T)(u \otimes v)= Su\otimes Tv$ for any $S\in\Bc(\Hc_1)$,  $T \in \Bc(\Hc_2)$, $u \in \Hc_1 $ and $v \in \Hc_2 $.

\smallskip

\noindent {\scriptsize $\bullet$} The linear span of $\{u \otimes v \colon u \in \Hc_1, \ v\in  \Hc_2\}$ is dense in $\Hc_1\otimes\Hc_2$.

\smallskip

\noindent {\scriptsize $\bullet$} The tensor product of two orthonormal bases is   an orthonormal basis.

\smallskip

\noindent {\scriptsize $\bullet$} The operator $S\otimes  T$ is invertible in  $\Bc(\Hc_1\otimes \Hc_2)$ if and only each operator, $S$ and $T$,  is invertible in $\Bc(\Hc_1)$ and $\Bc(\Hc_2)$ respectively.

\smallskip

\noindent {\scriptsize $\bullet$} The tensor product of two sequences  is a Riesz basis for $\Hc_1\otimes\Hc_2$ if and only if each sequence is a Riesz basis for its corresponding Hilbert space.

\smallskip

\noindent {\scriptsize $\bullet$} The tensor product of two Bessel sequences is a Bessel sequence for the corresponding Hilbert space.

\smallskip

\noindent {\scriptsize $\bullet$} The tensor product of two sequences  is a frame for $\Hc_1\otimes\Hc_2$ if and only if each sequence is a frame for its corresponding Hilbert space. }

\smallskip

Note that the fact that  $\Hc_1\otimes\Hc_2$ is  the completion of the linear span of $\Dc:=\{u \otimes v \colon u \in \Hc_1, \ v\in  \Hc_2\}$ yields that any operator of $\Bc(\Hc_1\otimes \Hc_2)$, in particular  the tensor product $S\otimes  T \in \Bc(\Hc_1\otimes \Hc_2)$, is determined by its values on $\Dc$. 

\smallskip

Finally, let $(X,\mu)$ and $(Y,\nu)$  denote two $\sigma$-finite measure spaces, then we have that $L^2(\mu)\otimes L^2(\nu)=L^2(\mu \times \nu)$ via the identification $(f\otimes g)(x, y)=f(x)g(y)$.

\section{Infinite-finite generators case}
\label{section2}
Let $\Hc_1$, $\Hc_2$ be two separable Hilbert spaces, and $U:\Hc_1 \longrightarrow \Hc_1$, $V:\Hc_2 \longrightarrow \Hc_2$ two unitary operators. Consider two elements $a\in \Hc_1$ and $b\in \Hc_2$ such that the sequence $\{U^n a\}_{n\in \ZZ}$ forms a Riesz sequence in $\Hc_1$ (see \cite[Theorem 2.1]{hector:14} for a necessary and sufficient condition), and there exists an $N\in \NN$ such that $V^N b=b$, being the set $\big\{b, Vb, V^2b,\dots,V^{N-1}b\big\}$ linearly independent in $\Hc_2$ (see \cite[Proposition 1]{garcia:15} for a necessary and sufficient condition). In the tensor product Hilbert space $\Hc_1\otimes \Hc_2$ we consider the closed subspace
\[
\Ac_{a,b}:=\overline{{\rm span}}_{\Hc_1\otimes \Hc_2} \big\{U^na\otimes V^pb\big\}_{n\in \ZZ;\,p=0,1,\dots, N-1}\,.
\]
Since the sequence $\big\{U^na\otimes V^pb\big\}_{n\in \ZZ;\,p=0,1,\dots, N-1}$ is a Riesz basis for the tensor product 
$\Ac_a \otimes \Ac_b$ of the $U$-invariant  subspace  $\Ac_a=\big\{\sum_{n\in \ZZ} a_n\, U^n a \,:\, \{a_n\}\in \ell^2(\ZZ)\big\}$ of $\Hc_1$ and the $V$-invariant subspace $\Ac_b=\big\{\sum_{p=0}^{N-1} b_p\, V^p b \, : \, b_p\in \CC\big\}$ of $\Hc_2$ we deduce that $\Ac_{a,b}=\Ac_a \otimes \Ac_b$,
and it can be described as
\[
\Ac_{a,b}=\Big\{\sum_{n\in \ZZ} \sum_{p=0}^{N-1} a_{np}\, U^na\otimes V^pb \,:\, \{a_{np}\}_n \in \ell^2(\ZZ)\,,\,p=0,1,\dots, N-1 \Big\}\,.
\]
We will refer to the vectors $\{a, b\}$ as the infinite-finite generators of the subspace $\Ac_{a,b}$ in $\Hc_1\otimes \Hc_2$.

\subsection*{The samples}

Fix $S=ss'$ elements $h_{jj'}:=h_{j1}\otimes h_{j'2} \in \Hc_1\otimes \Hc_2$, where $j=1,2, \dots,s$ and $j'=1,2, \dots,s'$. Consider two sampling periods $r$ and $\bar{r}$, where $r\in \NN$ and $\bar{r}$ is a divisor  of $N$; denote $\ell:=N/\bar{r}$. For each $x\in \Ac_{a,b}$ we introduce  the sequence of its generalized samples
\[
\Big\{\Lc_{jj'}x(rn,\bar{r}m) \Big\}_{\substack{n\in \ZZ;\,m=0,1, \dots, \ell-1 \\ j=1,2,\ldots,s;\,j'=1,2,\ldots,s'}}
\]
defined by
\begin{equation}
\label{samples}
\Lc_{jj'}x(rn,\bar{r}m):=\big\langle x, U^{rn}h_{j1}\otimes V^{\bar{r}m} h_{j'2}\big\rangle_{\Hc_1\otimes \Hc_2}\,,
\end{equation}
where $n\in \ZZ$, $m=0, 1,\dots, \ell-1$, $j=1,2,\dots, s$ and $j'=1,2, \dots,s'$. We will refer to any $\Lc_{jj'}$ as a $UV$-system acting on the subspace $\Ac_{a,b}$.

\subsection*{The isomorphism $\Tc^{UV}_{a,b}$}
Let $\ell^2_N(\ZZ)$ denote the space of all $N$-periodic sequences with inner product 
$\langle \mathbf{x}, \mathbf{y} \rangle_{\ell^2_N}=\sum_{p=0}^{N-1} x(p)\,\overline{y(p)}$, and its canonical basis as 
$\{\mathbf{e}_p\}_{p=0}^{N-1}$. This space is isomorphic to the euclidean space $\CC^N$; along the paper we will identify sequences in $\ell^2_N(\ZZ)$ with vectors in $\CC^N$: any vector in $\CC^N$ defines the terms from $0$ to $N-1$ of the corresponding sequence in $\ell^2_N(\ZZ)$.

We introduce the isomorphism $\Tc^{UV}_{a,b}$ which maps the canonical orthonormal basis $\big\{{\rm e}^{2\pi i nx}\otimes \mathbf{e}_p\big\}_{n\in \ZZ;\,p=0,1,\dots, N-1}$ for the Hilbert space $L^2(0,1)\otimes \ell^2_N(\ZZ)$ onto the Riesz basis 
$\big\{U^na\otimes V^pb \big\}_{n\in \ZZ;\,p=0,1,\dots, N-1}$ for $\Ac_{a,b}$. That is,
\begin{equation*}
\label{iso1}
\begin{array}[c]{ccll}
\Tc^{UV}_{a,b}:& L^2(0,1)\otimes \ell^2_N(\ZZ)& \longrightarrow & \Ac_{a,b}\\
      & {\rm e}^{2\pi i nx}\otimes \mathbf{e}_p & \longmapsto & U^na\otimes V^pb\,,
\end{array}
\end{equation*}
where $n\in \ZZ$ and $p=0,1,\dots, N-1$. 

\medskip 

It is clear  that $\Tc^{UV}_{a,b}=\Tc^U_a \otimes \Tc^V_{b,N}$, where $\Tc^U_a$ and $\Tc^V_{b,N}$ denote the isomorphisms
\[
\begin{array}[c]{ccll}
\Tc_a^U: & L^2(0, 1) & \longrightarrow & \Ac_a\\
        & {\rm e}^{2\pi i nx} & \longmapsto & U^n a  
\end{array} \quad \text{and} \quad       
        \begin{array}[c]{ccll}
\Tc^V_{b,N}: & \ell^2_N(\ZZ) & \longrightarrow & \mathcal{A}_b\\
        & \mathbf{e}_p & \longmapsto & V^p b\,.
\end{array}
\]
The following shifting property will be crucial later in obtaining our sampling formulas:
\begin{lema}
\label{shift}
Let $g\in L^2(0, 1)$ and $\mathbb{T}_0=\big(T(0), T(1), \dots, T(N-1)\big)\in \ell^2_N(\ZZ)$. For $M\in \ZZ$ and $1\le q\le N-1$ the shifting property
\[
\Tc^{UV}_{a,b}\big( g(x){\rm e}^{2\pi i Mx}\otimes \mathbb{T}_{N-q}\big)=\big(U^M\otimes V^q\big)\,\Tc^{UV}_{a,b}(g\otimes \mathbb{T}_0)
\]
holds, where $\mathbb{T}_{N-q}$ denotes the sequence obtained from $\mathbb{T}_0$ as 
\[
\mathbb{T}_{N-q}:=\big(T(N-q), T(N-q+1), \dots, T(N-q+N-1)\big)\,.
\]
\end{lema}
\begin{proof}
Indeed, having in mind the shifting properties: $\Tc^U_a (g(x){\rm e}^{2\pi i Mx})=U^M(\Tc^U_a g)$ for $\Tc^U_a$ (see \cite[Eq.(2.1)]{hector:14}) and 
$\Tc^V_{b,N} (\mathbb{T}_{N-q})=V^q(\Tc^V_{b,N} \mathbb{T}_0)$ for $\Tc^V_{b,N}$  (see \cite[Prop.2]{garcia:15}), and the properties of the tensor product of operators we have 
\[
\begin{split}
\Tc^{UV}_{a,b}\big( g(x){\rm e}^{2\pi i Mx}\otimes \mathbb{T}_{N-q}\big)&=\Tc^U_a (g(x){\rm e}^{2\pi i Mx})\otimes \Tc^V_{b,N} (\mathbb{T}_{N-q})=
U^M(\Tc^U_a g)\otimes V^q(\Tc^V_{b,N} \mathbb{T}_0) \\
&= \big(U^M\otimes V^q\big) \,(\Tc^U_a g \otimes \Tc^V_{b,N} \mathbb{T}_0)=\big(U^M\otimes V^q\big) \,\Tc^{UV}_{a,b}(g\otimes \mathbb{T}_0)\,.
\end{split}
\]
\end{proof}

\subsection*{An expression for the samples}
A suitable expression for the samples given in \eqref{samples} will allow us to obtain the reconstruction conditions of any 
$x$ in the subspace $\Ac_{a,b}$ from the sequence of its samples \eqref{samples}. Indeed, for 
$x=\sum_{k\in \ZZ} \sum_{p=0}^{N-1} a_{kp}\, U^ka\otimes V^pb$ in $\Ac_{a,b}$ we have
\[
\begin{split}
\Lc_{jj'}x(rn,\bar{r}m)&=\Big\langle \sum_{k\in \ZZ} \sum_{p=0}^{N-1} a_{kp}\, U^ka\otimes V^pb, U^{rn}h_{j1}\otimes V^{\bar{r}m}h_{j'2} \Big\rangle_{\Hc_1\otimes \Hc_2} \\
&=\sum_{k\in \ZZ} \sum_{p=0}^{N-1} a_{kp}\, \big\langle U^ka\otimes V^pb, U^{rn}h_{j1}\otimes V^{\bar{r}m}h_{j'2} \big\rangle_{\Hc_1\otimes \Hc_2}
\end{split}
\]
Now, consider $F:=\sum_{k\in \ZZ} \sum_{p=0}^{N-1} a_{kp}\, {\rm e}^{2\pi i kx}\otimes \mathbf{e}_p$, the element in $L^2(0,1)\otimes \ell^2_N(\ZZ)$ such that $\Tc^{UV}_{a,b} F=x$. Thus, Plancherel identity for orthonormal bases gives
\begin{equation*}
\begin{split}
\Lc_{jj'}x(rn,\bar{r}m)&=\Big\langle F, \sum_{k\in \ZZ} \sum_{p=0}^{N-1}\langle \overline{a, U^{rn-k} h_{j1}}\rangle_{\Hc_1}\, \langle \overline{b, V^{\bar{r}m-p} h_{j'2}}\rangle_{\Hc_2} \, {\rm e}^{2\pi i kx}\otimes \mathbf{e}_p\Big\rangle_{L^2(0,1)\otimes \ell^2_N(\ZZ)} \\ 
=&\Big\langle F, \big(\sum_{k\in \ZZ}\langle \overline{a, U^{rn-k} h_{j1}}\rangle_{\Hc_1}\,{\rm e}^{2\pi i kx} \big)\otimes \big( \sum_{p=0}^{N-1} \langle \overline{b, V^{\bar{r}m-p} h_{j'2}}\rangle_{\Hc_2}\,\mathbf{e}_p \big)\Big\rangle_{L^2(0,1)\otimes \ell^2_N(\ZZ)} \\ 
\end{split}
\end{equation*}
That is, for $n\in \ZZ$, $m=0, 1, \dots, \ell-1$, $j=1, 2, \dots, s$ and $j'=1, 2, \dots, s'$ we have got the following expression for the samples:
\begin{equation}
\label{sampexpression}
\Lc_{jj'}x(rn,\bar{r}m)=\Big\langle F,  \overline{g_j(x)}\,{\rm e}^{2\pi i rnx} \otimes \mathbb{G}_{j',m} \Big\rangle_{L^2(0,1)\otimes \ell^2_N(\ZZ)}
\end{equation}
where $\Tc^{UV}_{a,b} F=x$, the functions 
\begin{equation}
\label{gj}
g_j(x):=\sum_{k\in \ZZ} \langle a, U^k h_{j1} \rangle_{\Hc_1}\,{\rm e}^{2\pi i kx}\,, \quad j=1, 2, \dots, s
\end{equation}
belong to $L^2(0,1)$ (recall that $\{U^n a\}_{n\in \ZZ}$ is a Riesz sequence for $\Hc_1$), and
\[
\mathbb{G}_{j',m}:=\sum_{p=0}^{N-1} \langle \overline{b, V^{\bar{r}m-p} h_{j'2}}\rangle_{\Hc_2}\,\mathbf{e}_p=\sum_{p=0}^{N-1} \overline{R_{b,h_{j'2}}(N+p-\bar{r}m)}\, \mathbf{e}_p\,, \quad j'=1, 2, \dots, s'\,,
\]
where 
\[
R_{b,h_{j'2}}(q):=\langle V^q b, h_{j'2}\rangle_{\Hc_2}\,, \quad q\in \ZZ\,,
\] 
denotes the ($N$-periodic) {\em cross-covariance} sequence between the sequences  $\{V^qb\}_{q\in \ZZ}$ and $\{V^qh_{j'2}\}_{q\in \ZZ}$ in $\Hc_2$ (see \cite{kolmogorov:41}). Thus, we deduce the following result:
\begin{prop}
Any element $x\in \Ac_{a,b}$ can be recovered from the sequence of its samples 
$\Big\{\Lc_{jj'}x(rn,\bar{r}m) \Big\}_{\substack{n\in \ZZ;\,m=0,1, \dots, \ell-1 \\ j=1,2,\ldots,s;\,j'=1,2,\ldots,s'}}
$ in a stable way (by means of a frame expansion) if and only if the sequence $\Big\{\overline{g_j(x)}\,{\rm e}^{2\pi i rnx} \otimes \mathbb{G}_{j',m} \Big\}_{\substack{n\in \ZZ;\,m=0,1, \dots, \ell-1 \\ j=1,2,\ldots,s;\,j'=1,2,\ldots,s'}}$ is a frame for $L^2(0,1)\otimes \ell^2_N(\ZZ)$.
\end{prop}

Having in mind that a tensor product of two sequences is a frame in the tensor product if and only if the respective factors are frames, we only need a  characterization of the sequences 
$\big\{ \overline{g_j(x)}\,{\rm e}^{2\pi {\rm i} rnx} \big\}_{n\in \mathbb{Z};\,  j=1,2,\dots, s}$ and 
$\big\{ \mathbb{G}_{j',m}\big\}_{\substack{m=0,1,\ldots,\ell-1\\ j'=1,2,\ldots,s'}}$ as frames for $L^2 (0, 1)$ and $\ell^2_N(\ZZ)$ respectively. This study has been done in Refs. \cite{garcia:06,garcia:15} respectively. 

\medskip

\noindent $\bullet$ For the first one, consider the $s\times r$ matrix of  functions in $L^2(0, 1)$
\begin{equation}
\label{Gmatrix} 
\mathbf{G}(x):=
\begin{pmatrix} g_1(x)& g_1(x+\frac{1}{r})&\cdots&g_1(x+\frac{r-1}{r})\\
g_2(x)& g_2(x+\frac{1}{r})&\cdots&g_2(x+\frac{r-1}{r})\\
\vdots&\vdots&&\vdots\\
g_s(x)& g_s(x+\frac{1}{r})&\cdots&g_s(x+\frac{r-1}{r})
\end{pmatrix}= 
\bigg(g_j\Big(x+\frac{k-1}{r}\Big)\bigg)_{\substack{j=1,2,\ldots,s \\ k=1,2,\ldots, r}}
\end{equation}
and its related constants 
\[
\alpha_{\mathbf{G}}:=\einf_{x \in (0,1/r)}\lambda_{\min}[\mathbf{G}^*(x)\mathbf{G}(x)],\quad
\beta_{\mathbf{G}}:=\esup_{x \in (0,1/r)}\lambda_{\max}[\mathbf{G}^*(x)\mathbf{G}(x)]\,,
\]
where $\mathbf{G}^*(x)$ denotes the transpose conjugate of the matrix 
$\mathbf{G}(x)$, and $\lambda_{\min}$ (respectively $\lambda_{\max}$) the 
smallest (respectively the largest) eigenvalue of the positive semidefinite matrix 
$\mathbf{G}^*(x)\mathbf{G}(x)$. Observe that 
$0 \leq \alpha_{\mathbf{G}} \leq \beta_{\mathbf{G}} \leq \infty$.
Notice that in the definition of the matrix $\mathbf{G}(x)$ we are considering $1$-periodic 
extensions of the involved functions $g_j$, \ $j=1,2,\ldots,s$. 

A complete characterization of the sequence $\big\{ \overline{g_j(x)}\, {\rm e}^{2\pi {\rm i} rnx} \big\}_{n\in \mathbb{Z};\,  j=1,2,\dots, s}$ as a frame or a Riesz basis for $L^2(0, 1)$ is given in next lemma (see \cite[Lemma 3]{garcia:06} or \cite[Lemma 2]{garcia:08} for the proof):
\begin{lema}
\label{l2}
\label{caracterizacion}
For the functions $g_{j}\in L^2(0,1)$, $j=1,2,\ldots,s$, consider the associated matrix $\mathbf{G}(x)$ given in \eqref{Gmatrix}. Then, the following results hold:
\begin{enumerate}[(a)]
\item The sequence $\{\overline{g_{j}(x)}\,{\rm e}^{2\pi {\rm i} rnx}\}_{n\in\mathbb{Z};\,j=1,2,\ldots,s}$ is a complete system for $L^2(0,1)$ if and only if the rank of the matrix $\mathbf{G}(x)$ is $r$ a.e. in $(0,1/r)$.
\item The sequence
$\{\overline{g_{j}(x)}\,{\rm e}^{2\pi {\rm i} rnx}\}_{n\in\mathbb{Z};\,j=1,2,\ldots,s}$
is a Bessel sequence for  $L^2(0,1)$ if and only if $g_{j}\in
L^\infty(0,1)$ (or equivalently $\beta_{\mathbf{G}}<\infty$). In this
case, the optimal Bessel bound is $\beta_{\mathbf{G}}/r$. 
\item The sequence $\{\overline{g_{j}(x)}\,{\rm e}^{2\pi {\rm i} rnx}\}_{n\in\mathbb{Z};\,j=1,2,\ldots,s}$ is a frame for
$L^2(0,1)$ if and only if $0<\alpha_{\mathbf{G}}\le \beta_{\mathbf{G}}<\infty$. In this case, the optimal frame bounds are
$\alpha_{\mathbf{G}}/r$ and  $\beta_{\mathbf{G}}/r$.
\item  The sequence $\{\overline{g_{j}(x)}\,{\rm e}^{2\pi {\rm i} rnx}\}_{n\in\mathbb{Z};\,j=1,2,\ldots,s}$ is a Riesz basis for $L^2(0,1)$ if and only if it is a frame and $s=r$.
\end{enumerate}
\end{lema}

\medskip

\noindent $\bullet$ For the second one, consider the $s\ell \times N$ matrix of cross-covariances
\begin{equation}
\label{matrixcc}
\mathbf{R}_{\mathbf{b},\mathbf{h_2}}:= \begin{pmatrix}       
\mathbf{R_{b,h_{12}}} \\
\mathbf{R_{b,h_{22}}} \\  
\vdots \\
\mathbf{R_{b,h_{s'2}}}   \\
\end{pmatrix},
\end{equation}
where each $\ell \times N$ block $\mathbf{R_{b,h_{j'2}}}$, $j'=1,2,\dots,s'$, is given by
\[
{\scriptsize
\mathbf{R_{b,h_{j'2}}}=\begin{pmatrix}
R_{b,h_{j'2}}(0) & R_{b,h_{j'2}}(1) & \hdots & R_{b,h_{j'2}}(N-1) \\
      R_{b,h_{j'2}}(N-\bar{r}) & R_{b,h_{j'2}}(N-\bar{r}+1)& \hdots & R_{b,h_{j'2}}(2N-\bar{r}-1) \\
           \vdots & \vdots& \ddots & \vdots \\
                R_{b,h_{j'2}}(N-\bar{r}(\ell-1)) & R_{b,h_{j'2}}(N-\bar{r}(\ell-1)+1)& \hdots & R_{b,h_{j'2}}(N-\bar{r}(\ell-1)+N-1) \\
\end{pmatrix}.}
\]
In Ref. \cite{garcia:15} it was proved that:
\begin{lema}
\label{l3}
The sequence $\big\{ \mathbb{G}_{j',m}\big\}_{\substack{m=0,1,\ldots,\ell-1 \\ j'=1,2,\ldots,s'}}$ is a frame for 
$\ell^2_N(\ZZ)$ (or equivalently, a spanning set since we are in finite dimension) for $\ell^2_N(\ZZ)$ if and only $\text{rank\ $\mathbf{R}_{\mathbf{b},\mathbf{h_2}}$}=N$.
\end{lema}

Notice that, necessarily,  $s\geq r$ and $s'\geq \bar{r}$ and, consequently, the number $S=ss'$ of needed $UV$-systems $\Lc_{jj'}$ must be $S\geq r\bar{r}$.

\subsection{The sampling result}
\label{subsection1}
In case $\Big\{\overline{g_j(x)}\,{\rm e}^{2\pi i rnx} \otimes \mathbb{G}_{j',m} \Big\}_{\substack{n\in \ZZ;\,m=0,1, \dots, \ell-1 \\ j=1,2,\ldots,s;\,j'=1,2,\ldots,s'}}$ is a frame for $L^2(0,1)\otimes \ell^2_N(\ZZ)$ we need to describe the family of appropriate (for sampling purposes) dual frames. In Refs. \cite{garcia:06,garcia:15} suitable dual frames of the frames $\big\{ \overline{g_j(x)}\,{\rm e}^{2\pi {\rm i} rnx} \big\}_{n\in \mathbb{Z};\,  j=1,2,\dots, s}$ for $L^2(0, 1)$ and $\big\{ \mathbb{G}_{j',m}\big\}_{\substack{m=0,1,\ldots,\ell-1 \\ j'=1,2,\ldots,s'}}$ for $\ell^2_N(\ZZ)$ respectively, have been obtained. For a notational and  reading ease we recall them:

\medskip

\noindent {\scriptsize $\bullet$} For the first case, choose functions $h_j$ in $L^\infty(0, 1)$, $j=1,2,\dots,s$, such that
\begin{equation}
\label{h}
\big[h_1(x), h_2(x), \dots, h_s(x)\big] \mathbf{G}(x)=[1, 0, \dots, 0] \quad \text{a.e. in $(0, 1)$}\,.
\end{equation}
It was proven in \cite{garcia:06} that the sequence $\{r h_{j}(x)\,{\rm e}^{2\pi {\rm i} rnx}\}_{n\in\mathbb{Z};\,j=1,2,\ldots,s}$ is a dual frame of the frame 
$\{\overline{g_{j}(x)}\,{\rm e}^{2\pi {\rm i} rnx}\}_{n\in\mathbb{Z};\,j=1,2,\ldots,s}$ in $L^2(0, 1)$.

\medskip

All the possible choices in \eqref{h} are given by the first row of the $r\times s$ matrices given by
\begin{equation*}
\label{many}
\mathbf{H}_{\mathbf{U}}(x):=\mathbf{G}^\dag(x)+\mathbf{U}(x)\big[\mathbf{I}_s-\mathbf{G}(x)\mathbf{G}^\dag(x)\big]\,,
\end{equation*}
where $\mathbf{G}^\dag(x)$ denotes the Moore-Penrose pseudo-inverse  of $\mathbf{G}(x)$ given by 
\[
\mathbf{G}^\dag(x)=\big[\mathbf{G}^*(x)\,\mathbf{G}(x)\big]^{-1}\,\mathbf{G}^*(x)\,,
\]
$\mathbf{U}(x)$ is any $r\times s$ matrix with entries in $L^\infty(0, 1)$, and  $\mathbf{I}_s$ is the identity matrix of order $s$ (see \cite{penrose:55}). Notice that the entries of $\mathbf{G}^\dag(x)$  are essentially bounded in $(0, 1)$ since the functions $g_j$, $j=1,2,\dots, s$, 
and $\det^{-1}\big[\mathbf{G}^* (x)\, \mathbf{G}(x)\big]$ are essentially bounded in $(0, 1)$.

\medskip

\noindent {\scriptsize $\bullet$} For the second case, the $N$-periodic extensions of the columns $\big\{ \mathbb{H}_{j',m}\big\}_{\substack{m=0,1,\ldots,\ell-1 \\ j'=1,2,\ldots,s'}}$ of a left-inverse $\mathbf{H}_{\mathbf{S}}$ of the $s'\ell \times N$ matrix $\mathbf{R}_{\mathbf{b},\mathbf{h_2}}$ written as
\begin{equation*}
\label{columnsH}
\scriptsize{
\mathbf{H}_{\mathbf{S}}=\begin{pmatrix}
\vdots & \vdots & \vdots & \vdots & \vdots & \vdots & \vdots &\vdots &\vdots &\vdots \\
\mathbb{H}_{1,0} & \hdots & \mathbb{H}_{1,\ell-1} & \mathbb{H}_{2,0} & \hdots & \mathbb{H}_{2,\ell-1} & \hdots & \mathbb{H}_{s',0}&\hdots &\mathbb{H}_{s',\ell-1}\\
\vdots & \vdots & \vdots & \vdots & \vdots & \vdots & \vdots &\vdots &\vdots&\vdots\\
\end{pmatrix}Ê},
\end{equation*}
and constructed from a $\bar{r}\times s'\ell$ matrix $\mathbf{S}$ such that 
\begin{equation}
\label{s}
\mathbf{S}\,\mathbf{R}_{\mathbf{b},\mathbf{h_2}}=\big(\mathbf{I}_{\bar{r}}, \mathbf{O}_{\bar{r}\times (N-\bar{r})}\big)
\end{equation}
as in \cite[Section 3.1]{garcia:15}, form an appropriate dual frame of $\big\{ \mathbb{G}_{j',m}\big\}_{\substack{m=0,1,\ldots,\ell-1 \\ j'=1,2,\ldots,s'}}$ in 
$\ell^2_N(\ZZ)$ (see \cite{garcia:15} for the details about the construction of $\mathbf{H}_{\mathbf{S}}$). 

All the possible $\bar{r}\times s'\ell$ matrices $\mathbf{S}$ satisfying \eqref{s} are given by the first $\bar{r}$ rows of any any left-inverse $\mathbf{H}$ of the matrix $\mathbf{R}_{\mathbf{b},\mathbf{h_2}}$. All these left-inverses can be expressed as 
\[
\mathbf{H}=\mathbf{R}_{\mathbf{b},\mathbf{h_2}}^\dag+\mathbf{U}\big[\mathbf{I}_{s'\ell}-\mathbf{R}_{\mathbf{b},\mathbf{h_2}}\,\mathbf{R}_{\mathbf{b},\mathbf{h_2}}^\dag\big]\,,
\]
where $\mathbf{R}_{\mathbf{b},\mathbf{h_2}}^\dag$ denotes the Moore-Penrose pseudo-inverse of $\mathbf{R}_{\mathbf{b},\mathbf{h_2}}$, 
and $\mathbf{U}$ is any arbitrary $N\times s'\ell$ matrix.

\smallskip

Finally, one deduce that the sequence $\Big\{rh_j(x)\,{\rm e}^{2\pi i rnx} \otimes \mathbb{H}_{j',m} \Big\}_{\substack{n\in \ZZ;\,m=0,1, \dots, \ell-1 \\ j=1,2,\ldots,s;\,j'=1,2,\ldots,s'}}$ is a dual frame of $\Big\{\overline{g_j(x)}\,{\rm e}^{2\pi i rnx} \otimes \mathbb{G}_{j',m} \Big\}_{\substack{n\in \ZZ;\,m=0,1, \dots, \ell-1 \\ j=1,2,\ldots,s;\,j'=1,2,\ldots,s'}}$ in $L^2(0,1)\otimes \ell^2_N(\ZZ)$. Indeed:

\begin{lema}
Assume that $\{x_n^1\}$, $\{y_n^1\}$ are dual frames in $\Hc_1$ and $\{x_m^2\}$, $\{y_m^2\}$ are dual frames in $\Hc_2$. Then, 
$\{x_n^1\otimes x_m^2\}$ and $\{y_n^1\otimes y_m^2\}$ form a pair of dual frames in the tensor product $\Hc_1\otimes \Hc_2$.
\end{lema}
\begin{proof}Let $T_N^k$  and $G_N$,  $k=1,2$ and $N \in \mathbb{N}$, be  the bounded operators defined by 
$$
\begin{array}[c]{ccll}
T_N^k: & \Hc_k & \longrightarrow & \Hc_k\\
        & h_k & \longmapsto &  {\displaystyle \sum_{n=-N}^N \langle h_k, y_n^k \rangle \, x_n^k}\,,
\end{array}$$ 
and
$$      
        \begin{array}[c]{ccll}
G_N: & \Hc_1\otimes \Hc_2 & \longrightarrow & \Hc_1\otimes \Hc_2\\
        & h & \longmapsto &   {\displaystyle\sum_{m,n=-N}^N \langle h, y_n^1\otimes y_m^2 \rangle \, x_n^1\otimes y_m^2}\,.
\end{array}
$$
The fact that $\{x_n^k\}$, $\{y_n^k\}$ are dual frames in $\Hc_k$ implies that each sequence $\{T_N^k\}_{N\in \mathbb{N}}
$, $k=1,2$, is a bounded sequence (in fact, $\|T_N^k\|^2 \leq B_x^k B_y^k$ for every $N$ where $B_x^k$ and  $B_y^k$ are 
the upper frame bounds of $\{x_n^k\}$ and $\{y_n^k\}$, respectively) which converges in the strong operator topology of $
\Bc (\Hc_k)$ to the identity operator $I_{\Hc_k}$. Thus,  Lemma 2.3 in \cite{bourou:08} yields that 
$\{T_N^1\otimes T_N^2\}_{N\in \mathbb{N}}$ converges  in the strong operator topology of 
$\Bc (\Hc_1\otimes \Hc_2)$ to $I_{\Hc_1}\otimes I_{\Hc_2}= I_{\Hc_1\otimes \Hc_2}$. Since  
\begin{equation*}
\begin{split}
(T_N^1\otimes T_N^2)(h_1\otimes h_2) &= T_N^1(h_1) \otimes T_N^2(h_2)=
\Big(\sum_{n=-N}^N \langle h_1, y_n^1 \rangle \, x_n^1\Big) \otimes \Big( \sum_{m=-N}^N \langle h_2, y_m^2 \rangle \, x_m^2\Big) \\
&=  \sum_{m, n=-N}^N \langle h_1, y_n^1\rangle \,\langle h_2, y_m^2 \rangle\, \big(x_n^1  \otimes x_m^2 \big)
= G_N (h_1\otimes h_2)
\end{split}
\end{equation*}
for every $N \in \mathbb{N}$ and every $h_k \in \Hc_k$, $k=1,2$, then, 
$(T_N^1\otimes T_N^2)(h) = G_N(h)$   for every $h\in   \Hc^1\otimes \Hc^2$
and every $N \in \mathbb{N}$.  Hence $\{G_N\}_{N\in \mathbb{N}}$ converges  in the strong operator topology of $
\Bc (\Hc_1\otimes \Hc_2)$ to $ I_{\Hc_1\otimes \Hc_2}$, i.e., 
$$ h= \sum_{m,n} \,\langle h, y_n^1\otimes y_m^2 \rangle \, x_n^1\otimes y_m^2\, $$
for every  $h \in \Hc_1\otimes \Hc_2$, which concludes the proof.
\end{proof}

As a consequence, for each $x\in \Ac_{a,b}$ there exists a unique $F\in L^2(0,1)\otimes \ell^2_N(\ZZ)$ such that $\Tc^{UV}_{a,b} F=x$. This $F$ can be expressed as the frame expansion
\[
F=\sum_{j=1}^s \sum_{j'=1}^{s'}\sum_{n\in \ZZ} \sum_{m=0}^{\ell-1}\Lc_{jj'}x(rn,\bar{r}m)\, \big(rh_j(x)\,{\rm e}^{2\pi i rnx} \otimes \mathbb{H}_{j',m} \big) \quad \text{ in } L^2(0,1)\otimes \ell^2_N(\ZZ)\,.
\]
Then, applying the isomorphism $\Tc^{UV}_{a,b}$ and the shifting property in Lemma \ref{shift} (here it is the point where we are using that the proposed  dual frames are convenient for sampling purposes) one gets
\[
\begin{split}
x&= \sum_{j=1}^s \sum_{j'=1}^{s'}\sum_{n\in \ZZ} \sum_{m=0}^{\ell-1} \Lc_{jj'}x(rn,\bar{r}m)\, \Tc^{UV}_{a,b}\big(rh_j(x)\,{\rm e}^{2\pi i rnx} \otimes \mathbb{H}_{j',m} \big) \\
&= \sum_{j=1}^s \sum_{j'=1}^{s'} \sum_{n\in \ZZ} \sum_{m=0}^{\ell-1} \Lc_{jj'}x(rn,\bar{r}m)\, U^{rn} \big(\Tc^U_a rh_j \big) \otimes V^{\bar{r}m} \big(\Tc^V_{b,N} \mathbb{H}_{j',0}\big)\\
&=\sum_{j=1}^s \sum_{j'=1}^{s'}\sum_{n\in \ZZ} \sum_{m=0}^{\ell-1} \Lc_{jj'}x(rn,\bar{r}m)\, \big(U^{rn} \otimes V^{\bar{r}m} \big)\, ( c_j\otimes d_{j'})\,,\\
\end{split}
\]
where $c_j=\Tc^U_a (rh_j) \in \Ac_a$, $j=1, 2, \dots, s$ and $d_{j'}=\Tc^V_{b,N} (\mathbb{H}_{j',0})\in \Ac_b$, $j'=1, 2, \dots, s'$.

\medskip

Collecting the pieces we have obtained until now we prove the following result:
\begin{teo}
\label{samp1}
Let $h_{jj'}:=h_{j1}\otimes h_{j'2} \in \Hc_1\otimes \Hc_2$, $j=1,2, \dots,s$, $j'=1,2, \dots,s'$, and let $\Lc_{jj'}$ be the associated $UV$-system  giving the samples of any $x\in\Ac_{a,b}$ as in \eqref{samples}, $j=1,2, \dots,s$, $j'=1,2, \dots,s'$. Assume that the function $g_j$, $j=1,2,\dots,s$, given in \eqref{gj} belongs to $L^\infty(0, 1)$; or equivalently, that 
$\beta_{\mathbf{G}}<\infty$ for the associated $s\times r$ matrix $\mathbf{G}(x)$ defined in \eqref{Gmatrix}. The following statements are equivalent:
\begin{enumerate}[(a)]
\item $\alpha_{\mathbf{G}}>0$ and $\text{rank\ $\mathbf{R}_{\mathbf{b},\mathbf{h_2}}$}=N$.
\item There exist $S=ss'$ elements $c_j\otimes d_{j'}$ in the subspace $\mathcal{A}_{a,b}$, $j=1,2, \dots, s$, $j'=1,2, \dots, s'$, such that the sequence 
$\big\{\big(U^{rn} \otimes V^{\bar{r}m} \big)\, ( c_j\otimes d_{j'})\big\}_{\substack{n\in \ZZ;\,m=0,1, \dots, \ell-1 \\ j=1,2,\ldots,s;\,j'=1,2,\ldots,s'}}$ is a frame for $\mathcal{A}_{a,b}$, and for any 
$x\in \mathcal{A}_{a,b}$ the expansion
\begin{equation}
\label{framesampling1}
x=\sum_{j=1}^s \sum_{j'=1}^{s'}\sum_{n\in \ZZ} \sum_{m=0}^{\ell-1} \Lc_{jj'}x(rn,\bar{r}m)\, \big(U^{rn} \otimes V^{\bar{r}m} \big)\, ( c_j\otimes d_{j'}) \quad \text{in $\Hc_1\otimes \Hc_2$}\,,
\end{equation}
holds.
\item There exists a frame $\big\{C_{j,j',n,m}\big\}_{\substack{n\in \ZZ;\,m=0,1, \dots, \ell-1 \\ j=1,2,\ldots,s;\, j'=1,2,\ldots,s'}}$ for $\Ac_{a,b}$ such that, for each $x\in \Ac_{a,b}$ the expansion
\[
x=\sum_{j=1}^s \sum_{j'=1}^{s'} \sum_{n\in \ZZ} \sum_{m=0}^{\ell-1} \Lc_{jj'}x(rn,\bar{r}m)\, C_{j,j',n,m} \quad \text{in $\Hc_1\otimes \Hc_2$}\,,
\]
holds.
\end{enumerate}

\end{teo}
\begin{proof}
It only remains to prove that condition $(c)$ implies condition $(a)$. Indeed, applying the isomorphism $\big(\Tc^{UV}_{a,b}\big)^{-1}$ to the expansion given in 
$(c)$ one gets the frame expansion
\[
\big(\Tc^{UV}_{a,b}\big)^{-1}(x)=\sum_{j=1}^s \sum_{j'=1}^{s'}\sum_{n\in \ZZ} \sum_{m=0}^{\ell-1} \Lc_{jj'}x(rn,\bar{r}m)\, \big(\Tc^{UV}_{a,b}\big)^{-1} \big(C_{j,j',n,m}\big) \quad \text{in $L^2(0,1)\otimes \ell^2_N(\ZZ)$}\,.
\]
Having in mind \eqref{sampexpression}, and that the sequence  $\Big\{\overline{g_j(x)}\,{\rm e}^{2\pi i rnx} \otimes \mathbb{G}_{j',m} \Big\}_{\substack{n\in \ZZ;\,m=0,1, \dots, \ell-1 \\ j=1,2,\ldots,s;\, j'=1,2,\ldots,s'}}$ is a Bessel sequence  for $L^2(0,1)\otimes \ell^2_N(\ZZ)$, according to \cite[Lemma 5.6.2]{ole:03}, one gets that it is a frame for $L^2(0,1)\otimes \ell^2_N(\ZZ)$. In particular, this implies, via Lemmas \ref{l2} and \ref{l3}, that  $\alpha_{\mathbf{G}}>0$ and 
$\text{rank\ $\mathbf{R}_{\mathbf{b},\mathbf{h_2}}$}=N$.
\end{proof}

In case $s=r$ and $s'=\bar{r}$ we are in the Riesz bases setting: see statement $(d)$ in Lemma \ref{caracterizacion}; moreover, the square matrix $\mathbf{R}_{\mathbf{b},\mathbf{h_2}}$ must be invertible (see \cite[Corollary 4]{garcia:15}). In fact, the following corollary holds:
\begin{cor}
In addition to the hypotheses  of theorem above, assume that $s=r$ and $s'=\bar{r}$. The following statements are equivalent:
\begin{enumerate}
\item $\alpha_{\mathbf{G}}>0$ and the square matrix $\mathbf{R}_{\mathbf{b},\mathbf{h_2}}$ is invertible.
\item There exist $r\bar{r}$ unique elements $c_j\otimes d_{j'}$, $j=1,2, \dots, r$, $j'=1,2, \dots, \bar{r}$, in the subspace $\mathcal{A}_{a,b}$ such that the sequence 
$\Big\{\big(U^{rn} \otimes V^{\bar{r}m} \big)\, ( c_j\otimes d_{j'})\Big\}_{\substack{n\in \ZZ;\,m=0,1, \dots, \ell-1 \\ j=1,2,\ldots,r;\,j'=1,2,\ldots,\bar{r}}}$ is a Riesz basis for $\mathcal{A}_{a,b}$, and the expansion of any $x\in \mathcal{A}_{a,b}$ with respect to this basis is
\begin{equation*}
x=\sum_{j=1}^r \sum_{j'=1}^{\bar{r}}\sum_{n\in \ZZ} \sum_{m=0}^{\ell-1} \Lc_{jj'}x(rn,\bar{r}m)\, \big(U^{rn} \otimes V^{\bar{r}m} \big)\, ( c_j\otimes d_{j'}) \quad \text{in $\Hc_1\otimes \Hc_2$}\,.
\end{equation*}
\end{enumerate}
In case the equivalent conditions are satisfied, the vectors $c_j\otimes d_{j'}$, $j=1, 2, \dots, r$, $j'=1, 2, \dots, \bar{r}$, satisfy the interpolation property 
\begin{equation}
\label{interpolation}
\Lc_{kk'} \big(c_j\otimes d_{j'}\big) (rn,\bar{r}m)=\delta_{j,k}\,\delta_{j',k'}\,\delta_{n,0}\,\delta_{m,0}\,,
\end{equation}
whenever $n\in \ZZ$, $m=0, 1, \dots, \ell-1$, $j,k=1, 2, \dots, r$ and $j',k'=1, 2, \dots, \bar{r}$.
\end{cor}
\begin{proof}
The uniqueness of the expansion with respect to a Riesz basis gives the stated interpolation property \eqref{interpolation}.
\end{proof}

\subsection*{A representative example}
In $\Hc_1:=L^2(\RR)$ consider the shift operator $U: f(x) \mapsto f(x-1)$. Let  $\varphi \in L^2(\RR)$ be a function such that the sequence $\big\{\varphi(x-n)\big\}_{n\in \ZZ}$ is a Riesz sequence in $L^2(\RR)$ (for instance, the function $\varphi$ may be a $B$-spline). Consider also the Hilbert space $\Hc_2:=L^2(0, 1)$ of $1$-periodic functions and, for a fixed $N\in \NN$, the unitary operator $V_N: g(y)\mapsto g(y-\frac{1}{N})$. Let $\psi \in L^2(0, 1)$ be a function such that  the set $\big\{\psi, V_N \psi, \dots , V_N^{N-1} \psi  \big\}$ is linearly independent in $L^2(0, 1)$; obviously $V_N^N \psi=\psi$ (for instance,  choose a nonzero function $\psi \in L^2(0, 1)$ taking the value $0$ outside the interval $(0, 1/N)$). 

\medskip

In the tensor product $L^2(\RR) \otimes L^2(0, 1)=L^2\big(\RR\times (0, 1)\big)$ consider the closed subspace $\Ac_{\varphi,\psi}:=\Ac_\varphi \otimes \Ac_\psi$, i.e., the tensor product of the shift-invariant subspaces $\Ac_\varphi$ and $\Ac_\psi$ in $L^2(\RR)$. Given the functions $h_{j1}\in L^2(\RR)$, $j=1, 2, \dots, s$, and $h_{j'2}\in L^2(0, 1)$, $j'=1, 2, \dots, s'$, and fixing the sampling periods $r\in \NN$ and $\bar{r} | N$ (recall that $\ell=N/\bar{r}$), for each $f$ in $\Ac_{\varphi,\psi}$ we consider its samples defined by
\[
\Lc_{jj'}f(rn,\bar{r}m):=\int_{-\infty}^\infty \int_0^1 f(x,y)\, \overline{h_{j1}(x-rn)}\,\overline{h_{j'2}(y-m/\ell)}\,dxdy\,, 
\]
where $n\in \ZZ$ and $m=0, 1, \dots \ell-1$ (note that $\bar{r}m/N=m/\ell$). 

\medskip

Under the hypotheses in Theorem \ref{samp1} there will exist functions $S_j\in \Ac_\varphi$, $j=1, 2, \dots, s$, and $\Theta_{j'} \in \Ac_\psi$, $j'=1, 2, \dots, s'$, such that for any $f\in \Ac_{\varphi,\psi}$ the sampling expansion \eqref{framesampling1} reads:
\[
f(x,y)=\sum_{j=1}^s \sum_{j'=1}^{s'}\sum_{n\in \ZZ} \sum_{m=0}^{\ell-1} \Lc_{jj'}f(rn,\bar{r}m)\,S_j(x-rn)\,\Theta_{j'}(y-m/\ell) \quad \text{in $L^2\big(\RR\times (0, 1)\big)$}\,.
\]

In this particular example, adding some mild hypotheses we can also derive pointwise convergence in the above sampling formula. Indeed, assuming that the generators $\varphi$, $\psi$ are continuous functions on $\RR$ such that $\sum_{n\in \ZZ}|\varphi (x-n)|^2$ is bounded on $[0, 1]$ and since 
$\sum_{m=0}^{N-1}|\psi (y-m/N)|^2$ is bounded on $[0, 1/N]$, it is easy to deduce that any function $f$ in $\Ac_{\varphi,\psi}$ is a continuous function  defined by the pointwise sum 
\begin{equation}
\label{psum}
f(x, y)=\sum_{n\in \ZZ}\sum_{m=0}^{N-1} a_{n,m}\,\varphi (x-n)\,\psi (y-m/N)\quad \text{in $\RR\times [0, 1)$}\,.
\end{equation}
Besides, the subspace $\Ac_{\varphi,\psi}$ is a reproducing kernel Hilbert space (RKHS) since the evaluation functionals are bounded in $\Ac_{\varphi,\psi}$. Namely, using 
Cauchy-Schwarz's inequality in \eqref{psum} and Riesz basis definition, for each $(x, y)\in \RR\times [0, 1)$ we have
\[
|f(x, y)|^2 \le \frac{\|f\|^2}{A}\,\sum_{n\in \ZZ}|\varphi (x-n)|^2\,\sum_{m=0}^{N-1}|\psi (y-m/N)|^2\,,\quad f\in \Ac_{\varphi,\psi}\,,
\]
where $A$ denotes the lower Riesz bound of the Riesz basis $\big\{\varphi (x-n)\,\psi (y-m/N)\big\}_{n,m}$ for $\Ac_{\varphi,\psi}$.  The above inequality  shows that convergence in $L^2\big(\RR\times (0, 1)\big)$ implies pointwise convergence which is uniform on $\RR\times [0, 1)$. See, for instance, Ref. \cite{garcia2:15} for the relationship between sampling theory and RKHS's.

\section{Infinite-infinite generators case}
\label{section3}
Let $\Hc_1$, $\Hc_2$ be two separable Hilbert spaces, and $U:\Hc_1 \longrightarrow \Hc_1$, $V:\Hc_2 \longrightarrow \Hc_2$ two unitary operators. Consider two elements $a\in \Hc_1$ and $b\in \Hc_2$ such that the sequences $\{U^n a\}_{n\in \ZZ}$ and $\{V^m b\}_{m\in \ZZ}$ form a Riesz sequence in $\Hc_1$ and  in $\Hc_2$ respectively. In the tensor product Hilbert space $\Hc_1\otimes \Hc_2$ we consider its closed subspace
\[
\Ac_{a,b}:=\overline{{\rm span}}_{\Hc_1\otimes \Hc_2} \big\{U^na\otimes V^mb\big\}_{n,m\in \ZZ}\,.
\]
Since the sequence $\big\{U^na\otimes V^mb\big\}_{n,m\in \ZZ}$ is a Riesz basis for the tensor product 
$\Ac_a \otimes \Ac_b$ of the $U$-invariant subspace  $\Ac_a=\big\{\sum_{n\in \ZZ} a_n\, U^n a \,:\, \{a_n\} \in \ell^2(\ZZ)\big\}$ in $\Hc_1$ and the $V$-invariant subspace $\Ac_b=\big\{\sum_{m\in \ZZ} b_m\, V^m b \,:\, \{b_m\} \in \ell^2(\ZZ)\big\}$ in $\Hc_2$ we deduce that $\Ac_{a,b}=\Ac_a \otimes \Ac_b$
which can be described as
\[
\Ac_{a,b}=\Big\{\sum_{n,m\in \ZZ}  a_{nm}\, U^na\otimes V^mb \,:\, \{a_{nm}\} \in \ell^2(\ZZ^2) \Big\}\,.
\]
We will refer to the vectors $\{a, b\}$ as the infinite-infinite generators of the subspace $\Ac_{a,b}$ in $\Hc_1\otimes \Hc_2$.

\subsection*{The samples}
Fixed $S=ss'$ elements $h_{jj'}:=h_{j1}\otimes h_{j'2} \in \Hc_1\otimes \Hc_2$, $j=1,2, \dots,s$, $j'=1,2, \dots,s'$, we consider two sampling periods $r, \bar{r}\in \NN$. For each $x\in \Ac_{a,b}$ we introduce  the sequence of its generalized samples
\[
\Big\{\Lc_{jj'}x(rn,\bar{r}m) \Big\}_{\substack{n\in \ZZ;\,j=1,2,\ldots,s \\ m\in \ZZ;\,j'=1,2,\ldots,s'}}
\]
defined by
\begin{equation}
\label{samples2}
\Lc_{jj'}x(rn,\bar{r}m):=\big\langle x, U^{rn}h_{j1}\otimes V^{\bar{r}m} h_{j'2}\big\rangle_{\Hc_1\otimes \Hc_2}\,,
\end{equation}
where $n, m\in \ZZ$, $j=1,2,\dots, s$ and $j'=1,2,\dots, s'$. 

\subsection*{The isomorphism $\Tc^{UV}_{a,b}$}
In this case we introduce the isomorphism $\Tc^{UV}_{a,b}$ which maps the orthonormal basis $\big\{{\rm e}^{2\pi i nx}\otimes {\rm e}^{2\pi i mx}\big\}_{n, m\in \ZZ}$ for the Hilbert space $L^2(0,1)\otimes L^2(0,1)$ onto the Riesz basis $\big\{U^na\otimes V^mb \big\}_{n, m\in \ZZ}$ for $\Ac_{a,b}$. That is, $\Tc^{UV}_{a,b}=\Tc^U_a \otimes \Tc^V_b$, where $\Tc^U_a$ and $\Tc^V_b$ denote the isomorphisms
\[
\begin{array}[c]{ccll}
\Tc_a^U: & L^2(0, 1) & \longrightarrow & \Ac_a\\
        & {\rm e}^{2\pi i nx} & \longmapsto & U^n a  
\end{array} \quad \text{and} \quad       
        \begin{array}[c]{ccll}
\Tc^V_b: & L^2(0, 1) & \longrightarrow & \mathcal{A}_b\\
        & {\rm e}^{2\pi i mx} & \longmapsto & V^m b\,.
\end{array}
\]
Here, the shifting property reads:
\begin{equation}
\label{shift2}
\begin{split}
\Tc^{UV}_{a,b}\big( g(x){\rm e}^{2\pi i Mx}\otimes \widetilde{g}(x){\rm e}^{2\pi i Nx}\big)&=U^M\big(\Tc_a^U g \big)\otimes V^N\big( \Tc_b^V \widetilde{g}\big)\\
&=(U^M \otimes V^N)\, \Tc^{UV}_{a,b}(g\otimes \widetilde{g})\,,
\end{split}
\end{equation}
where $g, \widetilde{g} \in L^2 (0, 1)$ and $N, M\in \ZZ$. The proof of \eqref{shift2} goes in the same manner as in Lemma \ref{shift}.

\subsection*{An expression for the samples}
For  any $x=\sum_{k, p\in \ZZ}  a_{kp}\, U^ka\otimes V^pb$ in $\Ac_{a,b}$ we have
\[
\begin{split}
\Lc_{jj'}x(rn,\bar{r}m)&=\Big\langle \sum_{k, p\in \ZZ}  a_{kp}\, U^ka\otimes V^pb, U^{rn}h_{j1}\otimes V^{\bar{r}m}h_{j'2} \Big\rangle_{\Hc_1\otimes \Hc_2} \\
&=\sum_{k, p\in \ZZ}  a_{kp}\, \big\langle U^ka\otimes V^pb, U^{rn}h_{j1}\otimes V^{\bar{r}m}h_{j'2} \big\rangle_{\Hc_1\otimes \Hc_2}
\end{split}
\]
Now, consider $F:=\sum_{k, p\in \ZZ}  a_{kp}\, {\rm e}^{2\pi i kx}\otimes {\rm e}^{2\pi i px}$, the element in $L^2(0,1)\otimes L^2(0,1)$ such that 
$\Tc^{UV}_{a,b} F=x$. Thus, Plancherel identity for orthonormal bases gives
\begin{equation*}
\begin{split}
\Lc_{jj'}x(rn,\bar{r}m)&=\Big\langle F, \sum_{k, p\in \ZZ} \langle \overline{a, U^{rn-k} h_{j1}}\rangle_{\Hc_1}\, \langle \overline{b, V^{\bar{r}m-p} h_{j'2}}\rangle_{\Hc_2} \, {\rm e}^{2\pi i kx}\otimes {\rm e}^{2\pi i px}\Big\rangle_{L^2(0,1)\otimes L^2(0,1)} \\ 
=\Big\langle F&, \big(\sum_{k\in \ZZ}\langle \overline{a, U^{rn-k} h_{j1}}\rangle_{\Hc_1}\,{\rm e}^{2\pi i kx} \big)\otimes \big( \sum_{p\in \ZZ}\langle \overline{b, V^{\bar{r}m-p} h_{j'2}}\rangle_{\Hc_2}\,{\rm e}^{2\pi i px} \big)\Big\rangle_{L^2(0,1)\otimes L^2(0,1)} \\ 
\end{split}
\end{equation*}
That is, for $n, m\in \ZZ$, $j=1, 2, \dots, s$ and $j'=1, 2, \dots, s'$ we have got the following expression for the samples:
\begin{equation}
\label{sampexpression2}
\Lc_{jj'}x(rn,\bar{r}m)=\Big\langle F,  \overline{g_{j1}(x)}\,{\rm e}^{2\pi i rnx} \otimes \overline{g_{j'2}(x)}\,{\rm e}^{2\pi i \bar{r}mx} \Big\rangle_{L^2(0,1)\otimes L^2(0,1)}
\end{equation}
where, for $j=1, 2, \dots, s$ and $j'=1, 2, \dots, s'$, the functions 
\begin{equation}
\label{gj12}
g_{j1}(x):=\sum_{k\in \ZZ} \langle a, U^k h_{j1} \rangle_{\Hc_1}\,{\rm e}^{2\pi i kx} \quad \text{and}\quad g_{j'2}(x):=\sum_{k\in \ZZ} \langle b, V^k h_{j'2} \rangle_{\Hc_2}\,{\rm e}^{2\pi i kx}
\end{equation}
belong to $L^2(0,1)$. Thus, we deduce the following result:
\begin{prop}
Any element $x\in \Ac_{a,b}$ can be recovered from the sequence of its samples 
$\Big\{\Lc_{jj'}x(rn,\bar{r}m) \Big\}_{\substack{n\in \ZZ;\,j=1,2,\ldots,s \\ m\in \ZZ;\,j'=1,2,\ldots,s'}}$ in a stable way (by means of a frame expansion) if and only if the sequence $\Big\{\overline{g_{j1}(x)}\,{\rm e}^{2\pi i rnx} \otimes \overline{g_{j'2}(x)}\,{\rm e}^{2\pi i \bar{r}mx} \Big\}_{\substack{n\in \ZZ;\,j=1,2,\ldots,s \\ m\in \ZZ;\,j'=1,2,\ldots,s'}}$ is a frame for $L^2(0,1)\otimes L^2(0,1)$.
\end{prop}

The thesis of the above proposition is true if and only if $0<\alpha_{\mathbf{G}_1}\le \beta_{\mathbf{G}_1}<\infty$ and $0<\alpha_{\mathbf{G}_2}\le \beta_{\mathbf{G}_2}<\infty$, where $\mathbf{G}_1$ and $\mathbf{G}_2$ denote the $s\times r$ and $s'\times \bar{r}$ matrices defined in \eqref{Gmatrix} for $g_{j1}$, $j=1, 2, \dots, s$, and $g_{j'2}$, $j'=1, 2, \dots, s'$, respectively.

\subsection{The sampling result}
In case $\Big\{\overline{g_{j1}(x)}\,{\rm e}^{2\pi i rnx} \otimes \overline{g_{j'2}(x)}\,{\rm e}^{2\pi i \bar{r}mx} \Big\}_{\substack{n\in \ZZ;\,j=1,2,\ldots,s \\ m\in \ZZ;\,j'=1,2,\ldots,s'}}$ is a frame for $L^2(0,1)\otimes L^2(0,1)$, a family of dual frames is given by $\Big\{rh_{j1}(x)\,{\rm e}^{2\pi i rnx} \otimes \bar{r}h_{j'2}(x)\,{\rm e}^{2\pi i \bar{r}mx} \Big\}_{\substack{n\in \ZZ;\,j=1,2,\ldots,s \\ m\in \ZZ;\,j'=1,2,\ldots,s'}}$ where the functions $h_{j1}$, $j=1,2,\ldots,s$, and $h_{j'2}$, $j'=1,2,\ldots,s'$, in $L^\infty(0, 1)$, satisfy
\[
\begin{split}
\big[h_{11}(x), h_{21}(x), \dots, h_{s1}(x)\big] \mathbf{G}_1(x)&=[1, 0, \dots, 0]\,, \\
\big[h_{12}(x), h_{22}(x), \dots, h_{s'2}(x)\big] \mathbf{G}_2(x)&=[1, 0, \dots, 0] \quad \text{a.e. in $(0, 1)$}\,.
\end{split}
\]
Thus, for each $x\in \Ac_{a,b}$ there exists a unique $F\in L^2(0,1)\otimes L^2(0,1)$ such that $\Tc^{UV}_{a,b} F=x$. This $F$ can be expressed in $L^2(0,1)\otimes L^2(0,1)$ as the frame expansion
\[
F=\sum_{j=1}^s \sum_{j'=1}^{s'}\sum_{n\in \ZZ} \sum_{m\in \ZZ}\Lc_{jj'}x(rn,\bar{r}m)\, \Big(rh_{j1}(x)\,{\rm e}^{2\pi i rnx} \otimes \bar{r}h_{j'2}(x)\,{\rm e}^{2\pi i \bar{r}mx} \Big)\,. 
\]
Then, applying the isomorphism $\Tc^{UV}_{a,b}$ and  \eqref{shift2} one gets
\[
\begin{split}
x&= \sum_{j=1}^s \sum_{j'=1}^{s'}\sum_{n\in \ZZ} \sum_{m\in \ZZ}  \Lc_{jj'}x(rn,\bar{r}m)\, \Tc^{UV}_{a,b}\big(rh_{j1}(x)\,{\rm e}^{2\pi i rnx} \otimes \bar{r}h_{j'2}(x)\,{\rm e}^{2\pi i \bar{r}mx} \big) \\
&= \sum_{j=1}^s \sum_{j'=1}^{s'}\sum_{n\in \ZZ} \sum_{m\in \ZZ}  \Lc_{jj'}x(rn,\bar{r}m)\, U^{rn} \big(\Tc^U_a rh_{j1} \big) \otimes V^{\bar{r}m} \big(\Tc^V_b  \bar{r}h_{j'2}\big)\\
&=\sum_{j=1}^s \sum_{j'=1}^{s'}\sum_{n\in \ZZ} \sum_{m\in \ZZ}  \Lc_{jj'}x(rn,\bar{r}m)\, \big(U^{rn} \otimes V^{\bar{r}m} \big)\, ( c_{j1}\otimes d_{j'2})\,,\\
\end{split}
\]
where $c_{j1}=\Tc^U_a (rh_{j1}) \in \Ac_a$, $j=1, 2, \dots ,s$ and $d_{j'2}=\Tc^V_b (\bar{r}h_{j'2})\in \Ac_b$, $j'=1, 2, \dots , s'$. Next, we state the equivalent result to Theorem \ref{samp1}; its proof goes in the same manner:
\begin{teo}
\label{samp2}
Let $h_{jj'}:=h_{j1}\otimes h_{j'2} \in \Hc_1\otimes \Hc_2$, $j=1,2, \dots,s$ and $j'=1,2, \dots,s'$. Let $\Lc_{jj'}$ be the associated $UV$-system  giving the samples of any $x\in\Ac_{a,b}$ as in \eqref{samples2}, $j=1,2, \dots,s$ and $j'=1,2, \dots,s'$. Assume that the functions $g_{j1}$ and $g_{j'2}$, $j=1,2,\dots,s$ and $j'=1,2, \dots,s'$, given in \eqref{gj12} belongs to 
$L^\infty(0, 1)$; or equivalently, that $\beta_{\mathbf{G}_1}<\infty$ and $\beta_{\mathbf{G}_2}<\infty$ for the correspondent $s\times r$ and 
$s'\times \bar{r}$ matrices $\mathbf{G}_1(x)$ and $\mathbf{G}_2(x)$, respectively, defined in \eqref{Gmatrix}. The following statements are equivalent:
\begin{enumerate}[(a)]
\item $\alpha_{\mathbf{G}_1}>0$ and $\alpha_{\mathbf{G}_2}>0$.
\item There exist $S=ss'$ elements $c_j\otimes d_{j'}$ in the subspace $\mathcal{A}_{a,b}$, $j=1,2, \dots, s$ and $j'=1,2, \dots,s'$, such that the sequence 
$\big\{\big(U^{rn} \otimes V^{\bar{r}m} \big)\, ( c_j\otimes d_{j'})\big\}_{\substack{n\in \ZZ;\,j=1,2,\ldots,s \\ m\in \ZZ;\,j'=1,2,\ldots,s'}}$ is a frame for $\mathcal{A}_{a,b}$, and for any 
$x\in \mathcal{A}_{a,b}$ the expansion
\begin{equation}
\label{framesampling2}
x=\sum_{j=1}^s \sum_{j'=1}^{s'}\sum_{n\in \ZZ} \sum_{m\in \ZZ}  \Lc_{jj'}x(rn,\bar{r}m)\, \big(U^{rn} \otimes V^{\bar{r}m} \big)\, ( c_j\otimes d_{j'}) \quad \text{in $\Hc_1\otimes \Hc_2$}\,,
\end{equation}
holds.
\item There exists a frame $\big\{C_{j,j',n,m}\big\}_{\substack{n\in \ZZ;\,j=1,2,\ldots,s \\ m\in \ZZ;\,j'=1,2,\ldots,s'}}$ for $\Ac_{a,b}$ such that, for each $x\in \Ac_{a,b}$ the expansion
\[
x=\sum_{j=1}^s \sum_{j'=1}^{s'}\sum_{n\in \ZZ} \sum_{m\in \ZZ}  \Lc_{jj'}x(rn,\bar{r}m)\, C_{j,j',n,m} \quad \text{in $\Hc_1\otimes \Hc_2$}\,,
\]
holds.
\end{enumerate}
\end{teo}
In case $s=r$ and $s'=\bar{r}$ we are in the Riesz bases setting and theorem above admits the following corollary:
\begin{cor}
In addition to the hypotheses  of theorem above, assume that $s=r$ and $s'=\bar{r}$. The following statements are equivalent:
\begin{enumerate}
\item $\alpha_{\mathbf{G}_1}>0$ and $\alpha_{\mathbf{G}_2}>0$.
\item There exist $r\bar{r}$ unique elements $c_j\otimes d_{j'}$ in the subspace $\mathcal{A}_{a,b}$, $j=1,2, \dots, r$ and $j'=1,2,\ldots,\bar{r}$, such that the sequence 
$\big\{\big(U^{rn} \otimes V^{rm} \big)\, ( c_j\otimes d_{j'})\big\}_{\substack{n\in \ZZ;\,j=1,2,\ldots,r \\ m\in \ZZ;\,j'=1,2,\ldots,\bar{r}}}$ is a Riesz basis for 
$\mathcal{A}_{a,b}$, and the expansion of any $x\in \mathcal{A}_{a,b}$ with respect to this basis is
\begin{equation*}
x=\sum_{j=1}^r \sum_{j'=1}^{\bar{r}}\sum_{n\in \ZZ} \sum_{m\in \ZZ}  \Lc_{jj'}x(rn,\bar{r}m)\, \big(U^{rn} \otimes V^{rm} \big)\, ( c_j\otimes d_{j'}) \quad \text{in $\Hc_1\otimes \Hc_2$}\,.
\end{equation*}
\end{enumerate}
In case the equivalent conditions are satisfied, the vectors $c_j\otimes d_{j'}$, $j=1, 2, \dots, r$ and $j'=1,2,\ldots,\bar{r}$, satisfy the interpolation property 
\[
\Lc_{kk'} \big(c_j\otimes d_{j'}\big) (rn,rm)=\delta_{j,k}\,\delta_{j',k'}\,\delta_{n,0}\,\delta_{m,0}\,,
\]
whenever $n, m\in \ZZ$, $j,k=1, 2, \dots, r$ and $j',k'=1, 2, \dots, \bar{r}$.
\end{cor}

\subsection*{A representative example}
In $\Hc_1=\Hc_2:=L^2(\RR)$ consider $U=V$ the shift operator $f(x) \mapsto f(x-1)$. Let  $\varphi, \psi\in L^2(\RR)$ be two  functions such that the sequences $\big\{\varphi(x-n)\big\}_{n\in \ZZ}$ and $\big\{\psi(y-m)\big\}_{m\in \ZZ}$ are Riesz sequences for $L^2(\RR)$ (for instance, the functions 
$\varphi$, $\psi$ may be two $B$-splines). 

\medskip

In the tensor product $L^2(\RR) \otimes L^2(\RR)=L^2(\RR^2)$ consider the closed subspace $\Ac_{\varphi,\psi}:=\Ac_\varphi \otimes \Ac_\psi$, i.e., the tensor product of the shift-invariant subspaces $\Ac_\varphi$ and $\Ac_\psi$ of $L^2(\RR)$. Given the functions $h_{j1}, h_{j'2} \in L^2(\RR)$, $j=1, 2, \dots, s$ and $j'=1,2,\ldots,s'$, and fixing the sampling periods $r$ and $\bar{r}$ in $\NN$, for each $f$ in $\Ac_{\varphi,\psi}$ we consider its samples defined by
\[
\Lc_{jj'}f(rn,\bar{r}m):=\int_{-\infty}^\infty \int_{-\infty}^\infty f(x,y)\, \overline{h_{j1}(x-rn)}\,\overline{h_{j'2}(y-\bar{r}m)}dxdy\,, \quad n, m\in \ZZ\,.
\]
Under the hypotheses in Theorem \ref{samp2} there will exist functions $S_j\in \Ac_\varphi$, $j=1, 2, \dots, s$ and $\widetilde{S}_{j'} \in \Ac_\psi$, $j'=1, 2, \dots, s'$, such that for any $f\in \Ac_{\varphi,\psi}$ the sampling expansion \eqref{framesampling2} reads:
\[
f(x,y)=\sum_{j=1}^s  \sum_{j'=1}^{s'}\sum_{n\in \ZZ} \sum_{m\in \ZZ} \Lc_{jj'}f(rn,\bar{r}m)\,S_j(x-rn)\,\widetilde{S}_{j'}(y-\bar{r}m)\quad \text{ in $L^2(\RR^2)$}\,.
\]

As in the example of the infinite-finite case, assuming that the generators $\varphi$, $\psi$ are continuous functions on $\RR$ such that the sums $\sum_{n\in \ZZ}|\varphi (x-n)|^2$ and $\sum_{m\in \ZZ}|\psi (y-m)|^2$ are bounded on $[0, 1]$, it is easy to deduce that $\Ac_{\varphi,\psi}$ is a RKHS of continuous functions on $\RR^2$.  Furthermore, the convergence in $L^2\big(\RR^2\big)$ implies pointwise convergence which is uniform on $\RR^2$.

\section{Finite-finite generators case}
\label{section4}
Let $\Hc_1$, $\Hc_2$ be two separable Hilbert spaces, and $U:\Hc_1 \longrightarrow \Hc_1$, $V:\Hc_2 \longrightarrow \Hc_2$ two unitary operators. Consider two elements $a\in \Hc_1$ and $b\in \Hc_2$ such that, for some $N,M\in \NN$, $U^N a=a$ and $V^M b=b$ and the sets $\big\{a, Ua, U^2a,\dots,U^{N-1}b\big\}$ and $\big\{b, Vb, V^2b,\dots,V^{M-1}b\big\}$ are  linearly independent in $\Hc_1$ and $\Hc_2$ respectively. In the tensor product Hilbert space $\Hc_1\otimes \Hc_2$ we consider the finite subspace
\[
\Ac_{a,b}:={\rm span} \Big\{U^pa\otimes V^qb\Big\}_{\substack{p=0, 1, \dots, N-1\\ q=0, 1, \dots, M-1}}=\Big\{\sum_{p=0}^{N-1} \sum_{q=0}^{M-1} a_{pq}\, U^pa\otimes V^qb \Big\}\,.
\]
The subspace $\Ac_{a,b}$ coincides with the tensor product $\Ac_a \otimes \Ac_b$ of the finite subspaces $\Ac_a=\big\{\sum_{p=0}^{N-1} a_p\, U^p a, :\, a_p\in \CC\big\}\subset \Hc_1$ and $\Ac_b=\big\{\sum_{q=0}^{M-1} b_q\, V^m b\, :\, b_q\in \CC\big\}\subset \Hc_2$. We will refer to the vectors $\{a, b\}$ as the finite-finite generators of the subspace $\Ac_{a,b}$ in $\Hc_1\otimes \Hc_2$.

\subsection*{The samples}
Fixed $S=ss'$ elements $h_{jj'}:=h_{j1}\otimes h_{j'2} \in \Hc_1\otimes \Hc_2$, $j=1,2, \dots,s$ and $j'=1,2, \dots,s'$, we consider two sampling periods, 
$r$ a divisor  of $N$ and $\bar{r}$  a divisor of $M$; denote $\ell:=N/r$ and $\bar{\ell}:=M/\bar{r}$. For each $x\in \Ac_{a,b}$ we introduce  its 
$S\ell \bar{\ell}$ generalized samples
\[
\Big\{\Lc_{jj'}x(rn,\bar{r}m) \Big\}_{\substack{n=0, 1, \dots, \ell-1;\, j=1, 2, \dots,s\\ m=0, 1, \dots, \bar{\ell}-1;\, j'=1, 2, \dots,s'}}
\]
defined by
\begin{equation}
\label{samples3}
\Lc_{jj'}x(rn,\bar{r}m):=\big\langle x, U^{rn}h_{j1}\otimes V^{\bar{r}m} h_{j'2}\big\rangle_{\Hc_1\otimes \Hc_2}\,,
\end{equation}
where $n=0, 1, \dots, \ell-1$,  $m=0,1,\ldots, \bar{\ell}-1$, $j=1,2,\dots, s$ and $j'=1,2,\dots, s'$.

\subsection*{The isomorphism $\Tc^{UV}_{a,b}$}
As in the former cases, now we introduce the isomorphism $\Tc^{UV}_{a,b}$ which maps the orthonormal basis 
$\Big\{\mathbf{e}_p \otimes \mathbf{\widetilde{e}}_q\Big\}_{\substack{p=0, 1, \dots, N-1\\ q=0, 1, \dots, M-1}}$ for the Hilbert space 
$\ell^2_N(\ZZ)\otimes \ell^2_M(\ZZ)$ onto the basis $\Big\{U^pa\otimes V^qb \Big\}_{\substack{p=0, 1, \dots, N-1\\ q=0, 1, \dots, M-1}}$ for 
$\Ac_{a,b}$; here $\{\mathbf{e}_p\}_{p=0}^{N-1}$ and $\{ \mathbf{\widetilde{e}}_q\}_{q=0}^{M-1}$ denote, respectively, the canonical bases for 
$\ell^2_N(\ZZ)$ and $\ell^2_M(\ZZ)$. In other words, $\Tc^{UV}_{a,b}=\Tc^U_{a,N} \otimes \Tc^V_{b,M}$, where $\Tc^U_{a,N}$ and $\Tc^V_{b,M}$ denote the isomorphisms
\[
\begin{array}[c]{ccll}
\Tc_{a,N}^U: & \ell^2_N(\ZZ) & \longrightarrow & \Ac_a\\
        & \mathbf{e}_p & \longmapsto & U^p a  
\end{array} \quad \text{and} \quad       
        \begin{array}[c]{ccll}
\Tc^V_{b,M}: & \ell^2_M(\ZZ) & \longrightarrow & \mathcal{A}_b\\
        & \mathbf{\widetilde{e}}_q & \longmapsto & V^q b\,.
\end{array}
\]
Here, the shifting property reads:
\begin{equation}
\label{shift3}
\begin{split}
\Tc^{UV}_{a,b}\big(  \mathbb{T}_{N-p} \otimes  \mathbb{\widetilde{T}}_{M-q}\big)&=U^p\big(\Tc_{a,N}^U \mathbb{T}_0 \big)\otimes V^q\big( \Tc_{b,M}^V \mathbb{\widetilde{T}}_0\big)\\
&=(U^p \otimes V^q)\, \Tc^{UV}_{a,b}(\mathbb{T}_0 \otimes \mathbb{\widetilde{T}}_0)\,,
\end{split}
\end{equation}
where 
\[
\mathbb{T}_0=\big(T(0),T(1), \dots,T(N-1)\big),\, \mathbb{T}_{N-p}=\big(T(N-p), T(N-p+1), \dots, T(N-p+N-1)\big) 
\]
belong to $\ell^2_N(\ZZ)$ and $1\le p \le N-1$, and
\[
\mathbb{\widetilde{T}}_0=\big(\widetilde{T}(0), \widetilde{T}(1), \dots, \widetilde{T}(M-1)\big),\, \mathbb{\widetilde{T}}_{M-q}=\big(\widetilde{T}(M-q), \widetilde{T}(M-q+1),\dots,\widetilde{T}(M-q+M-1)\big)
\]
belong to $\ell^2_M(\ZZ)$ and  $1\le q \le M-1$. The proof of \eqref{shift3} goes in the same manner as in Lemma \ref{shift}.

\subsection*{An expression for the samples}
For  any $x=\sum_{p=0}^{N-1}\sum_{q=0}^{M-1}  a_{pq}\, U^pa\otimes V^qb$ in $\Ac_{a,b}$ we have
\[
\begin{split}
\Lc_{jj'}x(rn,\bar{r}m)&=\Big\langle \sum_{p=0}^{N-1}\sum_{q=0}^{M-1}  a_{pq}\, U^pa\otimes V^qb, U^{rn}h_{j1}\otimes V^{\bar{r}m}h_{j'2} \Big\rangle_{\Hc_1\otimes \Hc_2} \\
&=\sum_{p=0}^{N-1}\sum_{q=0}^{M-1}  a_{pq}\, \big\langle U^pa\otimes V^qb, U^{rn}h_{j1}\otimes V^{\bar{r}m}h_{j'2} \big\rangle_{\Hc_1\otimes \Hc_2}
\end{split}
\]
Now, consider $F:=\sum_{p=0}^{N-1}\sum_{q=0}^{M-1}  a_{pq}\, \mathbf{e}_p \otimes \mathbf{\widetilde{e}}_q$, the element in 
$\ell^2_N(\ZZ)\otimes \ell^2_M(\ZZ)$ such that $\Tc^{UV}_{a,b} F=x$. Thus, Plancherel identity for orthonormal bases gives
\begin{equation*}
\begin{split}
\Lc_{jj'}x(rn,\bar{r}m)&=\Big\langle F, \sum_{p=0}^{N-1}\sum_{q=0}^{M-1} \langle \overline{a, U^{rn-p} h_{j1}}\rangle_{\Hc_1}\, \langle \overline{b, V^{\bar{r}m-q} h_{j'2}}\rangle_{\Hc_2} \, \mathbf{e}_p \otimes \mathbf{\widetilde{e}}_q\Big\rangle_{\ell^2_N(\ZZ)\otimes \ell^2_M(\ZZ)} \\ 
&=\Big\langle F, \big(\sum_{p=0}^{N-1}\langle \overline{a, U^{rn-p} h_{j1}}\rangle_{\Hc_1}\,\mathbf{e}_p \big)\otimes \big( \sum_{q=0}^{M-1}\langle \overline{b, V^{\bar{r}m-q} h_{j'2}}\rangle_{\Hc_2}\,\mathbf{\widetilde{e}}_q \big)\Big\rangle_{\ell^2_N(\ZZ)\otimes \ell^2_M(\ZZ)} \\ 
\end{split}
\end{equation*}
That is, for $n=0, 1, \dots, \ell-1$, $m=0, 1, \dots, \bar{\ell}-1$, $j=1, 2, \dots, s$ and $j'=1, 2, \dots, s'$ we have got the following expression for the samples:
\begin{equation}
\label{sampexpression3}
\Lc_{jj'}x(rn,\bar{r}m)=\big\langle F,  \mathbb{G}_{j,n}^1 \otimes \mathbb{G}_{j',m}^2 \big\rangle_{\ell^2_N(\ZZ)\otimes \ell^2_M(\ZZ)}
\end{equation}
where
\[
\mathbb{G}_{j,n}^1:=\sum_{p=0}^{N-1} \langle \overline{a, U^{rn-p} h_{j1}}\rangle_{\Hc_1}\,\mathbf{e}_p=\sum_{p=0}^{N-1} \overline{R_{a,h_{j1}}(N+p-rm)}\, \mathbf{e}_p\,,
\]
and
\[
\mathbb{G}_{j',m}^2:=\sum_{q=0}^{M-1} \langle \overline{b, V^{\bar{r}m-q} h_{j'2}}\rangle_{\Hc_2}\,\mathbf{\widetilde{e}}_q=\sum_{q=0}^{M-1} \overline{R_{b,h_{j'2}}(M+q-\bar{r}m)}\, \mathbf{\widetilde{e}}_q\,.
\]
Here, $R_{a,h_{j1}}(k):=\langle U^k a, h_{j1}\rangle_{\Hc_1}$, $k\in \ZZ$,  denotes the ($N$-periodic) {\em cross-covariance} sequence between the sequences  $\{U^ka\}_{k\in \ZZ}$ and $\{U^kh_{j1}\}_{k\in \ZZ}$ in $\Hc_1$. Similarly, $R_{b,h_{j'2}}(k):=\langle V^k b, h_{j'2}\rangle_{\Hc_2}$, $k\in \ZZ$,
denotes the ($M$-periodic) {\em cross-covariance} sequence between the sequences  $\{V^kb\}_{k\in \ZZ}$ and $\{V^kh_{j'2}\}_{k\in \ZZ}$ in $\Hc_2$.
Thus, we deduce the following result:
\begin{prop}
Any element $x\in \Ac_{a,b}$ can be recovered from the sequence of its samples 
$\Big\{\Lc_{jj'}x(rn,\bar{r}m) \Big\}_{\substack{n=0, 1, \dots, \ell-1;\, j=1, 2, \dots,s\\ m=0, 1, \dots, \bar{\ell}-1;\, j'=1, 2, \dots,s'}}$ in a stable way (by means of a frame expansion) if and only if the sequence $\Big\{\mathbb{G}_{j,n}^1 \otimes \mathbb{G}_{j',m}^2 \Big\}_{\substack{n=0, 1, \dots, \ell-1;\, j=1, 2, \dots,s\\ m=0, 1, \dots, \bar{\ell}-1;\, j'=1, 2, \dots,s'}}$ is a frame for $\ell^2_N(\ZZ)\otimes \ell^2_M(\ZZ)$.
\end{prop}
The thesis of the above proposition is true if and only if the $s\ell \times N$ matrix $\mathbf{R}_{\mathbf{a},\mathbf{h_1}}$ (defined in \eqref{matrixcc} for 
$\mathbb{G}_{j,n}^1$) has rank $N$ and the $s'\bar{\ell} \times M$ matrix $\mathbf{R}_{\mathbf{b},\mathbf{h_2}}$ (defined in \eqref{matrixcc} for 
$\mathbb{G}_{j',m}^2$) has rank $M$. Notice that, necessarily, $s\geq r$ and $s'\geq r'$ and $ss'\ell \bar{\ell}\geq NM=\dim \big( \ell^2_N(\ZZ)\otimes \ell^2_M(\ZZ)\big)$.

\subsection{The sampling result}
In case the sequence $\Big\{\mathbb{G}_{j,n}^1 \otimes \mathbb{G}_{j',m}^2 \Big\}_{\substack{n=0, 1, \dots, \ell-1;\, j=1, 2, \dots,s\\ m=0, 1, \dots, \bar{\ell}-1;\, j'=1, 2, \dots,s'}}$ is a frame for $\ell^2_N(\ZZ)\otimes \ell^2_M(\ZZ)$, a family of appropriate dual frames is given by $\Big\{\mathbb{H}_{j,n}^1 \otimes \mathbb{H}_{j',m}^2 \Big\}_{\substack{n=0, 1, \dots, \ell-1;\, j=1, 2, \dots,s\\ m=0, 1, \dots, \bar{\ell}-1;\, j'=1, 2, \dots,s'}}$ where $\big\{\mathbb{H}_{j,n}^1\big\}_{\substack{n=0, 1, \dots, \ell-1\\ j=1, 2, \dots,s}}$ and $\big\{\mathbb{H}_{j',m}^2\big\}_{\substack{m=0, 1, \dots, \bar{\ell}-1\\ j'=1, 2, \dots,s'}}$ are appropriate dual frames of $\big\{\mathbb{G}_{j,n}^1\big\}_{\substack{n=0, 1, \dots, \ell-1\\ j=1, 2, \dots,s}}$ and $\big\{\mathbb{G}_{j',m}^2\big\}_{\substack{m=0, 1, \dots, \bar{\ell}-1\\ j'=1, 2, \dots,s'}}$ respectively. These dual frames are constructed as in Section \ref{subsection1} above (see \cite[Section 3.1]{garcia:15} for the details).

\medskip

Thus, for each $x\in \Ac_{a,b}$ there exists a unique $F\in \ell^2_N(\ZZ)\otimes \ell^2_M(\ZZ)$ such that $\Tc^{UV}_{a,b} F=x$. This $F$ can be expressed as the frame expansion
\[
F=\sum_{j=1}^s \sum_{j'=1}^{s'}\sum_{n=0}^{\ell-1}\sum_{m=0}^{\bar{\ell}-1} \Lc_{jj'}x(rn,\bar{r}m)\, \big(\mathbb{H}_{j,n}^1 \otimes \mathbb{H}_{j',m}^2\big) \quad \text{ in }  \ell^2_N(\ZZ)\otimes \ell^2_M(\ZZ)\,.
\]
Then, applying the isomorphism $\Tc^{UV}_{a,b}$ and  \eqref{shift3} one gets
\[
\begin{split}
x&= \sum_{j=1}^s \sum_{j'=1}^{s'}\sum_{n=0}^{\ell-1}\sum_{m=0}^{\bar{\ell}-1}  \Lc_{jj'}x(rn,\bar{r}m)\, \Tc^{UV}_{a,b}\big(\mathbb{H}_{j,n}^1 \otimes \mathbb{H}_{j',m}^2\big) \\
&= \sum_{j=1}^s \sum_{j'=1}^{s'}\sum_{n=0}^{\ell-1}\sum_{m=0}^{\bar{\ell}-1}  \Lc_{jj'}x(rn,\bar{r}m)\, U^{rn} \big(\Tc^U_{a,N} \mathbb{H}_{j,0}^1 \big) \otimes V^{\bar{r}m} \big(\Tc^V_{b,M}  \mathbb{H}_{j',0}^2\big)\\
&=\sum_{j=1}^s \sum_{j'=1}^{s'}\sum_{n=0}^{\ell-1}\sum_{m=0}^{\bar{\ell}-1}  \Lc_{jj'}x(rn,\bar{r}m)\, \big(U^{rn} \otimes V^{\bar{r}m} \big)\, ( c_{j}\otimes d_{j'})\,,\\
\end{split}
\]
where $c_{j}=\Tc^U_{a,N} \mathbb{H}_{j,0}^1 \in \Ac_a$, $j=1, 2, \dots , s$, and $d_{j'}=\Tc^V_{b,M}  \mathbb{H}_{j',0}^2\in \Ac_b$, $j'=1, 2, \dots , s'$. Next, we state the equivalent result to Theorem \ref{samp1}:
\begin{teo}
\label{samp3}
Let $h_{jj'}:=h_{j1}\otimes h_{j'2} \in \Hc_1\otimes \Hc_2$, $j=1,2, \dots,s$, $j'=1,2, \dots,s'$, and let $\Lc_{jj'}$ be the associated $UV$-system  giving the samples of any $x\in\Ac_{a,b}$ as in \eqref{samples3}, $j=1,2, \dots,s$, $j'=1,2, \dots,s'$. The following statements are equivalent:
\begin{enumerate}[(a)]
\item $\text{rank\ $\mathbf{R}_{\mathbf{a},\mathbf{h_1}}$}=N$ and $\text{rank\ $\mathbf{R}_{\mathbf{b},\mathbf{h_2}}$}=M$.
\item There exist $S=ss'$ elements $c_j\otimes d_{j'}$, $j=1,2, \dots, s$ and $j'=1,2, \dots,s'$, in the subspace $\mathcal{A}_{a,b}$ such that the sequence 
$\big\{\big(U^{rn} \otimes V^{\bar{r}m} \big)\, ( c_j\otimes d_{j'})\big\}_{\substack{n=0, 1, \dots, \ell-1;\, j=1, 2, \dots,s\\ m=0, 1, \dots, \bar{\ell}-1;\, j'=1, 2, \dots,s'}}$ is a frame for $\mathcal{A}_{a,b}$, and for any 
$x\in \mathcal{A}_{a,b}$ the expansion
\begin{equation}
\label{framesampling3}
x=\sum_{j=1}^s \sum_{j'=1}^{s'}\sum_{n=0}^{\ell -1} \sum_{m=0}^{\bar{\ell}-1}  \Lc_{jj'}x(rn,\bar{r}m)\, \big(U^{rn} \otimes V^{\bar{r}m} \big)\, ( c_j\otimes d_{j'})
\end{equation}
holds.
\item There exists a frame $\big\{C_{j,j',n,m}\big\}_{\substack{n=0, 1, \dots, \ell-1;\, j=1, 2, \dots,s\\ m=0, 1, \dots, \bar{\ell}-1;\, j'=1, 2, \dots,s'}}$ for $\Ac_{a,b}$ such that, for each $x\in \Ac_{a,b}$, the expansion
\[
x=\sum_{j=1}^s \sum_{j'=1}^{s'}\sum_{n=0}^{\ell -1} \sum_{m=0}^{\bar{\ell}-1} \Lc_{jj'}x(rn,\bar{r}m)\, C_{j,j',n,m}
\]
holds.
\end{enumerate}

\end{teo}

As in the previous cases, whenever  $s=r$ and $s'=\bar{r}$ in the above theorem we are in the Riesz bases setting necessarily, and condition $(a)$ says that both square matrices $\mathbf{R}_{\mathbf{a},\mathbf{h_1}}$ and $\mathbf{R}_{\mathbf{b},\mathbf{h_2}}$ are invertible.

\subsection*{A representative example}
In $\Hc_1:=\ell_N^2(\ZZ)$ consider the (circular) shift operator $U: x(p) \mapsto x(p-1)$, and  $a:=(1, 0, \dots, 0) \in \ell_N^2(\ZZ)$ which obviously satisfies $U^Na=a$ and $\Ac_a=\ell_N^2(\ZZ)$. Analogously, in  $\Hc_2:=\ell_M^2(\ZZ)$ consider the (circular) shift operator $V: y(q) \mapsto y(q-1)$, and  $b:=(1, 0, \dots, 0) \in \ell_M^2(\ZZ)$ which obviously satisfies $V^Mb=b$ and $\Ac_b=\ell_M^2(\ZZ)$. 

Given the sequences $\mathbf{h}_{j1}\in \ell_N^2(\ZZ)$, $j=1, 2, \dots, s$, and $\mathbf{h}_{j2}\in \ell_M^2(\ZZ)$, $j'=1, 2, \dots, s'$, and fixing the sampling periods $r\,|\,N$ and $\bar{r}\, | \,M$ (recall that 
$\ell=N/r$ and $\bar{\ell}=M/\bar{r}$), for each $\mathbf{x}$ in $\ell_N^2(\ZZ) \otimes \ell_M^2(\ZZ)=\ell_{N,M}^2(\ZZ^2)$, a double sequence in $\ZZ^2$, $N$-periodic in the first component and $M$-periodic in the second one, we consider its samples defined by
\[
\Lc_{jj'}\mathbf{x}(rn,\bar{r}m):=\sum_{p=0}^{N-1} \sum_{q=0}^{M-1}  \mathbf{x}(p,q)\, \overline{\mathbf{h}_{j1}(p-rn)}\,\overline{\mathbf{h}_{j'2}(q-\bar{r}m)}\,, 
\]
where $n=0,1, \dots, \ell-1$, $m=0, 1, \dots \bar{\ell}-1$, $j=1, 2, \dots, s$ and $j'=1, 2, \dots, s'$. 

\medskip

Under the hypotheses in Theorem \ref{samp3} there will exist sequences $\mathbf{c}_j\in \ell_N^2(\ZZ)$, $j=1, 2, \dots, s$, and $\mathbf{d}_{j'}\in \ell_M^2(\ZZ)$, $j'=1, 2, \dots, s'$, such that for any $\mathbf{x}\in \ell_{N,M}^2(\ZZ^2)$ the sampling expansion \eqref{framesampling3} reads:
\[
\mathbf{x}(p, q)=\sum_{j=1}^s \sum_{j'=1}^{s'}\sum_{n=0}^{\ell-1} \sum_{m=0}^{\bar{\ell}-1} \Lc_{jj'} \mathbf{x}(rn,\bar{r}m)\,\mathbf{c}_j(p-rn)\,\mathbf{d}_{j'}(q-\bar{r}m)\,,
\]
where $0\le p \le N-1$ and $0\le q \le M-1$.


\section{Discussion and conclusions}
\label{section6}

A sampling theory for tensor products of unitary invariant subspaces that allows to merge the cases of finitely/infinitely generated unitary invariant subspaces formerly studied in the mathematical literature is derived.   The involved samples are identified as frame coefficients in suitable tensor product spaces, thus it also allows to introduce  the several variables case in the formalism.

Alternatively, the theory developed here can also be considered as a theory of invariant subspaces in the tensor product of the corresponding Hilbert space with respect to the canonical unitary tensor product representation of the product group $\ZZ \otimes \ZZ_p$ defined from the corresponding factors (in the case of infinite-finite generators).   

In this sense, the results exhibited here point again in the direction recently started in the paper \cite{garcia3:15} and natural generalizations of them, like sampling theorems for tensor products of invariant subspaces with respect to unitary representations of finite or infinite groups on each factor, can be addressed using similar ideas.

Reduction techniques corresponding to the decomposition of tensor products of unitary representations into their irreducible components with respect to proper subgroups of the product group, could also be used to simplify the construction of samples, i.e., to find adapted frames.   This, as well as other applications beyond classical telecommunications all involving sampling of states of quantum systems,  will be the subject of further research.

\medskip

\noindent{\bf Acknowledgments:} This work has been supported by the grant MTM2014-54692-P from the Spanish {\em Ministerio de Econom\'{\i}a y Competitividad (MINECO)}.

\vspace*{0.3cm}

\end{document}